\documentclass{svjour3}                     % onecolumn (standard format)
\smartqed  % flush right qed marks, e.g. at end of proof
\usepackage{setspace}
\usepackage{algorithm}
\usepackage{algorithmic}

\usepackage{hyperref}
\usepackage{enumerate}
\usepackage{amssymb}
\usepackage{booktabs}
\usepackage{graphicx,fancyhdr}
\usepackage{graphics,color}
\usepackage{subfigure}%并列的子图形
\usepackage{booktabs}
\usepackage{graphicx}
\usepackage{amsmath}
\numberwithin{equation}{section}
\numberwithin{theorem}{section}
\numberwithin{lemma}{section}
\numberwithin{remark}{section}

\usepackage{mathptmx}
\begin{document}

\title{ Error estimates for backward fractional Feynman-Kac equation with non-smooth initial data}

%\titlerunning{Short form of title}        % if too long for running head

\author{Jing Sun$^{1}$, Daxin Nie$^{1}$, Weihua Deng$^{*,1}$
}

%\authorrunning{Short form of author list} % if too long for running head

\institute{
$^{*}$Corresponding author. E-mail: dengwh@lzu.edu.cn            \\
$^{1}$School of Mathematics and Statistics, Gansu Key Laboratory of Applied Mathematics and Complex Systems, Lanzhou University, Lanzhou 730000, P.R. China \\
}
%E-mail: Eli.Barkai@biu.ac.il}

\date{Received: date / Accepted: date}
% The correct dates will be entered by the editor

\maketitle

\begin{abstract}
In this paper, we are concerned with the numerical solution for the backward fractional Feynman-Kac equation with non-smooth initial data.
 %and this model is derived in [Carmi, Turgeman and Barkai, J. Stat. Phys. {\bf141}, 1071--1092, 2010].
Here we first provide the regularity estimate of the solution. And then we use the backward Euler and second-order backward difference convolution quadratures to approximate the Riemann-Liouville fractional substantial derivative and get the first- and second-order convergence in time. The finite element method is used to discretize the Laplace operator with the optimal convergence rates. Compared with the previous works for the backward fractional Feynman-Kac equation, the main advantage of the current discretization is that we don't need the assumption on the regularity of the solution in temporal and spatial directions. Moreover, the error estimates of the time semi-discrete schemes and the fully discrete schemes are also provided. Finally, we perform the numerical experiments to verify the effectiveness of the presented algorithms.

% \PACS{PACS code1 \and PACS code2 \and more}
% \subclass{MSC code1 \and MSC code2 \and more}
\end{abstract}

%\newpage

\noindent{\section{Introduction}}

The Feynman-Kac equation describes the distribution of the functionals of the trajectories of the particles, where the functional is defined as $A=\int^t_0U[x(\tau)]d\tau$ with $x(t)$ being a trajectory of a particle and $U(x)$ a prescribed function depending on specific applications \cite{Kac1949}. There are two kinds of Feynman-Kac equations: one is for the forward Feynman-Kac equation, governing the joint probability density of the functional and position; and another one is for the backward equation, just focusing on the distribution of the functionals. If the particles are with power-law waiting time and/or jump length distribution(s), the governing equations for the distribution of the functionals are so-called fractional Feynman-Kac equations \cite{Agmon1984,Carmi2010,Wang2018}, since the fractional substantial derivative is involved in the equations. More generalizations of the Feynman-Kac equations include the models governing  the distribution of the functionals of the particles undergoing the reaction and diffusion processes and of the particles with  multiple internal states \cite{Hou2018,Xu2018,Xu2018-2}.

%The fractional Feynman-Kac equation  describes the distribution of  functional for the particles with power-law waiting time or jump length, the functional of which is defined as $A=\int^t_0U[x(\tau)]d\tau$, where $x(t)$ is a trajectory of non-Brownian particle and $U(x)$ is a prescribed function depending on specific applications \cite{Agmon1984,Carmi2010,Wang2018}. There are two kinds fractional Feynman-Kac equations: one is for the forward fractional Feynman-Kac equation, governing the joint probability density of the functional and position, and another one is for the backward one, just focusing the distribution of the functional.
%%As we all know, there are two kinds of the fractional Feynman-Kac equations, i.e., the forward and backward fractional Feynman-Kac equations, and the backward fractional Feynman-Kac equation refers to the operators depend on the initial position $x_0$ and the fractional derivative operator appears to the left of the diffusion operator. With the deep research of the anomalous diffusion, there are many works for deriving the forward and backward fractional Feynman-Kac equations under the different physical contexts. For example, \cite{Hou2018,Xu2018,Xu2018-2} describe the distribution of  functional of the particles undergoing the reaction and diffusion processes and  multiple internal states, respectively, to derive the different type of fractional Feynman-Kac equations.
Here we solve the following backward fractional Feynman-Kac equation, presented in \cite{Carmi2010}, describing the functional distribution of the particles with power-law waiting time, i.e.,

\begin{equation}\label{eqretosol}
\left\{
\begin{aligned}
&\frac{\partial G(x_0,\rho,t)}{\partial t}=\,_0D^{1-\alpha,x_0}_t\Delta G(x_0,\rho,t)-\rho U(x_0)G(x_0,\rho,t), \quad(x_0,t)\in \Omega\times(0,T],\\
&G(x_0,\rho,0)=G_0(x_0),\qquad\qquad\qquad\qquad\qquad\qquad\qquad\qquad\quad  x_0\in \Omega,\\
&G(x_0,\rho,t)=0,\qquad\qquad\quad\qquad\qquad\qquad\qquad\qquad\qquad\qquad (x_0,t)\in \partial \Omega\times(0,T],
\end{aligned}
\right.
\end{equation}
where $G(x_0,\rho,t)=\int_{0}^{\infty}G(x_0,A,t)e^{-\rho A}d A$ and $G(x_0, A, t)$ is the joint probability density function of finding the particle on $A$ at time $t$ with the initial position of the particle at $x_0$;  $\rho$ is the Fourier pair of $A$; $\alpha\in(0,1)$;  $\Delta$ stands for Laplace operator;  $\Omega$ is a bounded domain and $U(x_{0})$ is assumed to be bounded in $\bar{\Omega}$ in this paper; $T$ is a fixed final time; $_0D^{1-\alpha,x_0}_t$  denotes the Riemann-Liouville fractional substantial derivative, whose definition \cite{Li2015} is
\begin{equation}
\begin{aligned}
_{0}D^{1-\alpha,x_0}_tG(x_0,\rho,t)=&\frac{1}{\Gamma(\alpha)}\left[\frac{\partial}{\partial t}+\rho U(x_0)\right]\int^t_{0}(t-\xi)^{\alpha-1}e^{-(t-\xi)\rho U(x_0)}G(x_0,\rho,\xi)d\xi\\
=&e^{-t\rho U(x_0)}~_0D^{1-\alpha}_t(e^{t\rho U(x_0)}G(x_0,\rho,t)),
\end{aligned}
\end{equation}
where $_{0}D^{\alpha}_{t}$ denotes the Riemann-Liouville fractional derivative with the definition \cite{Podlubny1999}
\begin{equation*}
_{0}D^{\alpha}_{t}G(x_0,\rho,t)=\frac{1}{\Gamma(1-\alpha)}\frac{\partial}{\partial t}\int^t_{0}(t-\xi)^{-\alpha}G(x_0,\rho,\xi)d\xi, \qquad\quad \alpha\in(0,1).
\end{equation*}

 So far there have been many works for fractional partial differential equations, including the finite difference method, finite element method, spectral method, and so on \cite{Acosta2018,Bazhlekova2015,chen2009,Cheng2015,Deng2008,Deng2013,Ervin2006,li2012,sun2006}, but there are relatively less researches on solving fractional Feynman-Kac equation numerically \cite{Chen2018,Deng2015,Deng:17,Deng2017,Nie2019}. The main reasons are that fractional substantial derivative is a  time-space coupled non-local operator and the equation covers the complex parameters which bring about many challenges on regularity and numerical analyses. To our best knowledge, numerical approximation on fractional substantial derivative is given in \cite{Chen2015}; Ref. \cite{Deng2015} numerically solves the forward and backward fractional Feynman-Kac equations with the assumptions that the solution is regular, $U(x_0)$ is a positive constant, and $Re(\rho)>0$ ($Re (\rho)$ means the real part of $\rho$); Ref. \cite{Deng2017} presents the $H^1$ error estimate for the backward fractional Feynman-Kac equation with $U(x_0)>0$ and $Re(\rho)>0$; Ref. \cite{Deng:17} provides an efficient time-stepping method to solve the forward fractional Feynman-Kac equation and makes error analysis in the measure norm. In this paper, we use the finite element method in space and convolution quadrature introduced in \cite {Lubich1988,Lubich1988-2} in time to solve the backward fractional Feynman-Kac equation \eqref{eqretosol}. The main contributions are as follows.
 \begin{itemize}
\item We first provide Sobolev regularity for the solution of Eq. \eqref{eqretosol}, i.e.,  Theorem \ref{thmreg} gives that the solution $G(x_0,\rho,t)\in \dot{H}^2(\Omega)$ when $U(x_0)$ is bounded in $\bar{\Omega}$ and $G_0(x_0)\in L^2(\Omega)$. Compared with the previous works \cite{Chen2018,Deng2015,Deng2017}, we construct numerical scheme without any assumption on the regularity of solution in temporal and spatial directions.
\item Then we modify the approximation of the  Riemann-Liouville fractional derivative got by convolution quadrature  to approximate the Riemann-Liouville fractional substantial derivative, which skillfully overcome the trouble brought by the non-commutativity of the Riemann-Liouville fractional derivative and  $e^{-t\rho U(x_0)}$ in error estimate for fully discrete scheme, i.e.,  $e^{-t\rho U(x_0)}~_0D^{\alpha}_t\neq ~_0D^{\alpha}_te^{-t\rho U(x_0)}$ in Eq. \eqref{eqretosoleq2}.

\item Next, a suitable modify based on the Laplace transform representation of solution is presented to guarantee the accuracy of second-order backward difference scheme \eqref{equsemischemeO2} (see Sec. 3).

\item Besides, motivated by the error estimate in space in \cite{Bazhlekova2015,Jin2016}, a general idea is to get the estimate of the difference between  $((\beta_{\tau,1}(z))^\alpha+A)(\beta_{\tau,1}(z))^{\alpha-1}G_0(x_0)$ and $((\beta_{\tau,1}(z))^\alpha+A_h)(\beta_{\tau,1}(z))^{\alpha-1}P_hG_0(x_0)$ (for their detailed definitions, see Sec. 3 and Sec. 4). Generally, the sufficient regularity on $U(x_0)$ is required to ensure the accuracy of the approximation.
    %But in fact, we need to require $U(x_0)$ is regular enough to keep the accuracy of discretization.
    Here, we use $P_h(e^{-t\rho U(x_0)}G_0(x_0))$ for the fully discrete scheme \eqref{equsfullscheme} in Sec. 4  instead of
 $e^{-t\rho U(x_0)}P_hG_0(x_0)$,  which weakens the requirement of regularity on $U(x_0)$ to keep the accuracy of the numerical scheme.

% which reduces the regularity assumption on $U(x_0)$ and also keeps the accuracy of the numerical scheme.

 \item Finally, we provide a complete error analysis for the proposed numerical scheme and obtain the optimal convergence rates in $L^2$- and $H^1$-norm.
\end{itemize}

 The rest of the paper is organized as follows. We first provide some preliminaries and a regularity estimate for the solution of Eq. \eqref{eqretosol} in Sec. 2. Section 3 presents the  approximation of the Riemann-Liouville fractional substantial derivative by backward Euler and second-order backward difference convolution quadratures and gives the error estimates of the time semi-discrete schemes. In Sec. 4, we use the finite element method to discretize the Laplace operator and provide the error estimate for the fully discrete scheme with the non-smooth initial data. In Sec. 5, we verify the effectiveness of the algorithm by numerical experiments. We conclude the paper with some discussions in the last section.

\noindent{\section{Preliminaries}}

First, we introduce $A=-\Delta$ with a zero Dirichlet boundary condition. For any $ r\geq  0 $, denote the space $ \dot{H}^r(\Omega)=\{v\in L^2(\Omega): A^{\frac{r}{2}}v\in L^2(\Omega) \}$ with the norm \cite{Thomee2006}
\begin{equation*}
\|v\|^2_{\dot{H}^r(\Omega)}=\sum_{j=1}^{\infty}\lambda_j^r(v,\varphi_j)^2,
\end{equation*}
where $ {(\lambda_j,\varphi_j)} $ are the eigenvalues ordered non-decreasingly and the corresponding eigenfunctions (normalized in the $ L^2(\Omega) $ norm) of operator $A$. Thus $ \dot{H}^0(\Omega)=L^2(\Omega) $, $\dot{H}^1(\Omega)=H^1_0(\Omega)$, and $\dot{H}^2(\Omega)=H^2(\Omega)\bigcap H^1_0(\Omega)$.
 For $\kappa>0$ and $\pi/2<\theta<\pi$, we define sectors $\Sigma_{\theta}$ and $\Sigma_{\theta,\kappa}$ in the complex plane $\mathbb{C}$ as
\begin{equation*}
	\begin{aligned}
		&\Sigma_{\theta}=\{z\in\mathbb{C}\setminus \{0\},|\arg z|\leq \theta\}, \\
		&\Sigma_{\theta,\kappa}=\{z\in\mathbb{C}:|z|\geq\kappa,|\arg z|\leq \theta\},\\
	\end{aligned}
\end{equation*}
and the contour $\Gamma_{\theta,\kappa}$ is defined by
\begin{equation*}
	\Gamma_{\theta,\kappa}=\{z\in\mathbb{C}: |z|=\kappa,|\arg z|\leq \theta\}\cup\{z\in\mathbb{C}: z=r e^{\pm \mathbf{i}\theta}: r\geq \kappa\},
\end{equation*}
oriented with an increasing imaginary part, where $\mathbf{i}$ denotes the imaginary unit and $\mathbf{i}^2=-1$. Then we denote  $\|\cdot\|$ as the operator norm from $L^2(\Omega)$ to $L^2(\Omega)$ and define $G(t)$ and $G_0$ as $G(x_0,\rho,t)$ and $G_0(x_0)$ respectively in the following.
Throughout this paper, $C$ denotes a generic positive constant, whose value may differ at each occurrence; and let $\epsilon>0$ arbitrary small.

Similar to the skill used in \cite{Chen2018,Deng2015,Deng2017}, Eq. \eqref{eqretosol} can also be converted into
\begin{equation}\label{eqretosoleq}
\left\{
\begin{aligned}
& \,^C_0D^{\alpha,x_0}_tG(x_0,\rho,t)=\Delta G(x_0,\rho,t), \quad\qquad\qquad\qquad\qquad(x_0,t)\in \Omega\times(0,T],\\
&G(x_0,\rho,0)=G_0(x_0),\qquad\qquad\qquad\qquad\qquad\qquad\qquad  x_0\in \Omega,\\
&G(x_0,\rho,t)=0,\qquad\qquad\quad\qquad\qquad\qquad\qquad\qquad\quad\,\, (x_0,t)\in \partial \Omega\times(0,T],
\end{aligned}
\right.
\end{equation}
where $^C_{0}D^{\alpha,x_0}_t$ denotes Caputo fractional substantial derivative defined by \cite{Li2015}
\begin{equation*}
\begin{aligned}
~^C_{0}D^{\alpha,x_0}_tG(x_0,\rho,t)=e^{-t\rho U(x_0)}~^C_0D^{\alpha}_t(e^{t\rho U(x_0)}G(x_0,\rho,t)), \qquad \alpha\in(0,1),
\end{aligned}
\end{equation*}
and $~^C_{0}D^{\alpha}_{t}$ means the Caputo fractional derivative with its definition \cite{Podlubny1999}
\begin{equation*}
~^C_{0}D^{\alpha}_{t}G(x_0,\rho,t)=\frac{1}{\Gamma(1-\alpha)}\int^t_{0}(t-\xi)^{-\alpha}\frac{\partial}{\partial \xi}G(x_0,\rho,\xi)d\xi, \qquad \alpha\in(0,1).
\end{equation*}

Then we recall the Laplace transform for the fractional substantial derivative.
\begin{lemma}[\cite{Li2015}]\label{procapLap}
	The Laplace transform of the Riemann-Liouville  fractional substantial derivative with $\alpha\in(0,1)$ is given by
	\begin{equation*}
	\widetilde{~_0D^{\alpha,x_0}_{t}G}(z)=(\beta(z,x_0))^\alpha \tilde{G}(z),
	\end{equation*}
	and the Laplace transform of the Caputo fractional substantial derivative with $\alpha\in(0,1)$ is given by
	\begin{equation*}
		\widetilde{~^C_0D^{\alpha,x_0}_{t}G}(z)=(\beta(z,x_0))^\alpha \tilde{G}(z)-(\beta(z,x_0))^{\alpha-1}G(0),
	\end{equation*}
	where $\beta(z,x_0)=(z+\rho U(x_0))$ and `$\,\tilde{~}$' stands for taking the Laplace transform. And in the following we denote $\beta(z)$ as $\beta(z,x_0)$.
\end{lemma}
According to Lemma \ref{procapLap}, the solution of Eq. $\eqref{eqretosoleq}$ can be written as
\begin{equation}\label{equsolrep}
	\tilde{G}(z)=((\beta(z))^{\alpha}+A)^{-1}(\beta(z))^{\alpha-1}G_{0}.
\end{equation}
\begin{remark}
	By the definition of $\beta(z)$, it is easy to see that
	\begin{equation*}
		\begin{aligned}
		&\beta(z) A\neq A\beta(z),\quad A((\beta(z))^{\alpha}+A)^{-1}\neq((\beta(z))^{\alpha}+A)^{-1} A,\\
		&((\beta(z))^{\alpha}+A)^{-1}(\beta(z))^{\alpha-1}\neq (\beta(z))^{\alpha-1}((\beta(z))^{\alpha}+A)^{-1}.
		\end{aligned}
	\end{equation*}
\end{remark}

Before we provide the regularity estimate for the solution of Eq. \eqref{eqretosoleq}, the following lemma about $\beta(z)$ is also needed.
\begin{lemma}[\cite{Deng:17}]\label{lemmaBeta}
	Let $U(x_0)$ be bounded in $\bar{\Omega}$. By choosing $\theta \in \left(\frac{\pi}{2},\pi\right)$ sufficiently close to $\frac{\pi}{2}$
	and $\kappa>0$ sufficiently large (depending on the value $|{{\rho}}|\|U(x_0)\|_{L^{\infty}(\bar{\Omega})}$), we have the following results:
	\begin{enumerate}[(1)]
		\item For all $x\in \Omega$ and ${{z}}\in \Sigma_{\theta,\kappa}$, we have $\beta({{z}}) \in \Sigma_{\frac{3\pi}{4},\frac{\kappa}{2}}$, and
		\begin{equation}\label{chapter4section2_2prop1conc1}
		C_1|{{z}}|\leq|\beta({{z}})|\leq C_2|{{z}}|,
		\end{equation}
		where $C_1$ and $C_2$ denote two positive constants.
		So $\beta({{z}})^{1-\alpha}$ and $\beta({{z}})^{\alpha-1}$ are both   analytic function of ${{z}}\in \Sigma_{\theta,\kappa} $.
		
		\item The operator $((\beta({{z}}))^\alpha+A)^{-1}:L^2(\Omega)\rightarrow L^2(\Omega)$ is well-defined, bounded, and analytic with respect to $z\in \Sigma_{\theta,\kappa}$, satisfying
		\begin{equation}\label{chapter4section2_2prop1conc21}
		\|A((\beta(z))^\alpha+A)^{-1}\|\leq C~~~~~ for~all ~~{{z}} \in \Sigma_{\theta,\kappa},
		\end{equation}
		\begin{equation}\label{chapter4section2_2prop1conc22}
		\|((\beta({{z}}))^\alpha+A)^{-1}\|\leq C|{{z}}|^{-\alpha}~~~for~all~~{{z}} \in \Sigma_{\theta,\kappa}.
		\end{equation}
	\end{enumerate}
\end{lemma}
%\begin{proof}
%	The detailed proof can be referred to \cite{Deng:17}.
%\end{proof}

\begin{theorem}\label{thmreg}
	Assume $U(x_0)$ is bounded in $\bar{\Omega}$. If $G_{0}\in L^2(\Omega)$ and $G(t)$ is the solution of Eq. \eqref{eqretosoleq}, then we have the estimate
	\begin{equation*}
	\|G(t)\|_{\dot{H}^{\sigma}(\Omega)}\leq Ct^{-\sigma\alpha/2}\|G_0\|_{L^2(\Omega)}, \quad \sigma\in[0,2].
	\end{equation*}
\end{theorem}
\begin{proof}
	Taking inverse Laplace transform and $L^2(\Omega)$ norm on both sides of \eqref{equsolrep}, according to Lemma \ref{lemmaBeta}, we have
	\begin{equation*}
		\begin{aligned}
			\|G(t)\|_{L^2(\Omega)}\leq& C\left \|\int_{\Gamma_{\theta,\kappa}}e^{zt}((\beta(z))^\alpha+A)^{-1}(\beta(z))^{\alpha-1}G_0dz\right \|_{L^2(\Omega)}\\
			\leq& C\int_{\Gamma_{\theta,\kappa}}|e^{zt}|\left \|((\beta(z))^\alpha+A)^{-1}(\beta(z))^{\alpha-1}\right \|\|G_0\|_{L^2(\Omega)}|dz|\\
			\leq& C\int_{\Gamma_{\theta,\kappa}}|e^{zt}| |z|^{-1}\|G_0\|_{L^2(\Omega)}|dz|\\
			\leq& C\left(\int_{\kappa}^{\infty}e^{r\cos(\theta)t}r^{-1}dr+\int_{-\theta}^{\theta}e^{\kappa\cos(\varphi)t}d\varphi\right)\|G_0\|_{L^2(\Omega)}\\
			\leq& C\|G_0\|_{L^2(\Omega)}.
		\end{aligned}
	\end{equation*}
Applying $A$ on both sides of \eqref{equsolrep}, taking inverse Laplace transform, and acting $L^2(\Omega)$ norm on both sides, from Lemma \ref{lemmaBeta}, there is
\begin{equation*}
\begin{aligned}
\|AG(t)\|_{L^2(\Omega)}\leq& C\left \|\int_{\Gamma_{\theta,\kappa}}e^{zt}A((\beta(z))^\alpha+A)^{-1}(\beta(z))^{\alpha-1}G_0dz\right \|_{L^2(\Omega)}\\
\leq& C\int_{\Gamma_{\theta,\kappa}}|e^{zt}|\left \|A((\beta(z))^\alpha+A)^{-1}(\beta(z))^{\alpha-1}\right \|\|G_0\|_{L^2(\Omega)}|dz|\\
\leq& C\int_{\Gamma_{\theta,\kappa}}|e^{zt}| |z|^{\alpha-1}\|G_0\|_{L^2(\Omega)}|dz|\\
\leq& C\left(\int_{\kappa}^{\infty}e^{r\cos(\theta)t}r^{\alpha-1}dr+\int_{-\theta}^{\theta}e^{\kappa\cos(\varphi)t}\kappa^{\alpha}d\varphi\right)\|G_0\|_{L^2(\Omega)}\\
\leq& Ct^{-\alpha}\|G_0\|_{L^2(\Omega)}.
\end{aligned}
\end{equation*}
Using interpolation properties \cite{Deng2017} leads to
\begin{equation*}
	\|G(t)\|_{\dot{H}^{\sigma}(\Omega)}\leq Ct^{-\sigma\alpha/2}\|G_0\|_{L^2(\Omega)},\quad \sigma\in[0,2].
\end{equation*}
	
\end{proof}

\section{Temporal discretization and error analysis}
In this section, we first use the backward Euler and second-order backward difference convolution quadratures introduced in \cite{Lubich1988,Lubich1988-2} to discretize the Riemann-Liouville fractional substantial derivative and get the first- and second-order schemes in time. Then we provide the complete error analysis.

 Let the time step size $\tau=T/N$ with $N\in\mathbb{N}$, $t_i=i\tau$, $i=0,1,\ldots,N$, and $0=t_0<t_1<\cdots<t_N=T$. Firstly, we use the relationship between Caputo fractional derivative and Riemann-Liouville fractional derivative
\begin{equation*}
\,^C_0D^{\alpha}_{t}u(t)=\,_0D^{\alpha}_{t}(u(t)-u(0))~~{\rm with}~~  \alpha\in(0,1)
\end{equation*}
%we have
%\begin{equation*}
%\,^C_0D^{\alpha,x_0}_tu(t)=e^{-\rho U(x_0)t}\,_0D^{\alpha}_{t}(e^{\rho U(x_0)t}u(t)-u(0)).
%\end{equation*}
to reformulate Eq. \eqref{eqretosoleq} with Riemann-Liouville fractional substantial derivative, i.e.,
\begin{equation}\label{eqretosoleq2}
\left\{
\begin{aligned}
& \,_0D^{\alpha,x_0}_tG(x_0,\rho,t)+A G(x_0,\rho,t)=e^{-\rho U(x_0)t}\,_0D^{\alpha}_{t}G(x_0,\rho,0), \,\,\,\,\,\, (x_0,t)\in \Omega\times(0,T],\\
&G(x_0,\rho,0)=G_0(x_0),\qquad\qquad\qquad\qquad\qquad\qquad\qquad\qquad x_0\in \Omega,\\
&G(x_0,\rho,t)=0,\qquad\qquad\qquad\quad\quad\qquad\qquad\qquad\qquad\qquad\,\,\, (x_0,t)\in \partial \Omega\times(0,T].
\end{aligned}
\right.
\end{equation}

\subsection{Backward Euler scheme and error estimate}
We use backward Euler convolution quadrature to discretize the time fractional substantial derivative and get the first-order accuracy in time. Introduce $G^n$ as the numerical approximation of solution $G(x_0,\rho,t_n)$. Then we can obtain the temporal semi-discrete scheme
\begin{equation}\label{equsemischeme}
\left\{
\begin{aligned}
	&\sum_{i=0}^{n-1}d^{\alpha,1}_{i}e^{-t_{i}\rho U(x_0)}G^{n-i}+A G^n=e^{-t_{n}\rho U(x_0)}\sum_{i=0}^{n-1}d^{\alpha,1}_{i}G^0,\quad\,\, x_0\in \Omega,\quad n=1,\ldots,N,\\	
	&G^0=G_0,\qquad\qquad\qquad\qquad\qquad\qquad\qquad\qquad\qquad\quad  x_0\in \Omega,\\
	&G^n(x_0)=0,\qquad\qquad\quad\qquad\qquad\qquad\qquad\qquad\qquad\quad x_0\in \partial \Omega,\quad n=1,\ldots,N,
	\end{aligned}
	\right.
\end{equation}
where
\begin{equation}\label{equdefd}
	(\delta_{\tau,1}(\zeta))^{\alpha}=\sum_{i=0}^{\infty}d^{\alpha,1}_{i}\zeta^i
\end{equation}
and
\begin{equation*}
	\delta_{\tau,1}(\zeta)=\frac{1-\zeta}{\tau }.
\end{equation*}
%\begin{remark}
%	This discretization
%\end{remark}
\iffalse
Then, we have following properties about $\delta_{\tau,1}(e^{-z\tau})$.

\begin{lemma}[\cite{Deng:17}]\label{lemmaprodeltatau}
Let $\theta\in(\frac{\pi}{2},\pi)$ sufficiently close to $\frac{\pi}{2}$, the following estimates hold
\begin{equation*}
	\begin{aligned}
		&C_1|z|\leq \left|\frac{1-e^{-\tau z}}{\tau}\right|\leq C_2|z|, \qquad \forall z\in \Sigma_{\theta},~ |Im(z)|\leq \frac{\pi}{\tau},\\
		&\left|\frac{1-e^{-\tau z}}{\tau}-z\right|\leq C\tau|z|^{2},\qquad \forall z\in \Sigma_{\theta},~ |Im(z)|\leq \frac{\pi}{\tau},\\
		&\left|\left (\frac{1-e^{-\tau z}}{\tau}\right )^\alpha-z^\alpha\right|\leq C\tau|z|^{2},\qquad \forall z\in \Sigma_{\theta},~ |Im(z)|\leq \frac{\pi}{\tau},\\
		&\left(\frac{1-e^{-\tau z}}{\tau}\right)^\alpha\in \Sigma_{\theta},\qquad \forall z\in \Sigma_{\theta},~ |Im(z)|\leq \frac{\pi}{\tau},		
	\end{aligned}
\end{equation*}	
where $Im(z)$ means the imaginary part of $z$.
\end{lemma}
\fi
Multiplying $\zeta^n$ on both sides of the first formula of \eqref{equsemischeme} and summing $n$ from $1$ to $\infty$ lead to
\begin{equation}\label{equsemimid}
	\sum_{n=1}^{\infty}\sum_{i=0}^{n-1}d^{\alpha,1}_{i}e^{-t_{i}\rho U(x_0)}G^{n-i}\zeta^n+\sum_{n=1}^{\infty}A G^n\zeta^n=\sum_{n=1}^{\infty}e^{-t_{n}\rho U(x_0)}\sum_{i=0}^{n-1}d^{\alpha,1}_{i}G^0\zeta^n.
\end{equation}
Simple calculation implies
\begin{equation*}
	(\delta_{\tau,1}(e^{-\tau\rho U(x_0)}\zeta))^{\alpha}\sum_{n=1}^{\infty}G^n\zeta^n+A\sum_{n=1}^{\infty}G^n\zeta^n=(\delta_{\tau,1}(e^{-\tau\rho U(x_0)}\zeta))^{\alpha}\sum_{n=1}^{\infty}e^{-t_{n}\rho U(x_0)}G^0\zeta^n,
\end{equation*}
which is followed by $\eqref{equdefd}$. Furthermore, we have
\begin{equation*}
(\delta_{\tau,1}(e^{-\tau\rho U(x_0)}\zeta))^{\alpha}\sum_{n=1}^{\infty}G^n\zeta^n+A\sum_{n=1}^{\infty}G^n\zeta^n=(\delta_{\tau,1}(e^{-\tau\rho U(x_0)}\zeta))^{\alpha-1}\frac{G^0e^{-\tau\rho U(x_0)}\zeta}{\tau}.
\end{equation*}
%Introducing $\beta_{\tau,1}(z)=\delta_{\tau,1}(e^{-\tau(z+\rho U(x_0))})$ and taking $\zeta=e^{-z\tau}$ yield
%\begin{equation*}
%	\sum_{n=0}^{\infty}G^ne^{-zt_n}=((\beta_{\tau,1}(z))^{\alpha}+A)^{-1}(\beta_{\tau,1}(z))^{\alpha-1}\frac{G^0e^{-\tau(z+\rho U(x_0))}}{\tau}.
%\end{equation*}
Using Cauchy's integral formula yields
\begin{equation}\label{betatau}
	\begin{aligned}
		G^n=&\frac{1}{2\pi\mathbf{i}}\int_{|\zeta|=\xi_\tau}\zeta^{-n-1}((\delta_{\tau,1}(e^{-\tau\rho U(x_0)}\zeta))^{\alpha}+A)^{-1}(\delta_{\tau,1}(e^{-\tau\rho U(x_0)}\zeta))^{\alpha-1}\frac{G^0e^{-\tau\rho U(x_0)}\zeta}{\tau}d\zeta\\
		=&\frac{1}{2\pi\mathbf{i}}\int_{\Gamma^\tau}e^{zt_n}((\beta_{\tau,1}(z))^{\alpha}+A)^{-1}(\beta_{\tau,1}(z))^{\alpha-1}e^{-\tau(z+\rho U(x_0))}G^0dz,
	\end{aligned}
\end{equation}
where $\xi_\tau=e^{-\tau(\kappa+1)}$, $\Gamma^\tau=\{z=\kappa+1+\mathbf{i}y:y\in\mathbb{R}~{\rm and}~|y|\leq \pi/\tau\}$, $\beta_{\tau,1}(z)=\delta_{\tau,1}(e^{-\tau(z+\rho U(x_0))})$, and the second equality follows by taking $\zeta=e^{-z\tau}$. Deforming the contour $\Gamma^\tau$ to
$\Gamma^\tau_{\theta,\kappa}=\{z\in \mathbb{C}:\kappa\leq |z|\leq\frac{\pi}{\tau\sin(\theta)},|\arg z|=\theta\}\cup\{z\in \mathbb{C}:|z|=\kappa,|\arg z|\leq\theta\}$, one has
\begin{equation}\label{equsemisolrep}
	G^n=\frac{1}{2\pi\mathbf{i}}\int_{\Gamma_{\theta,\kappa}^{\tau}}e^{zt_n}((\beta_{\tau,1}(z))^{\alpha}+A)^{-1}(\beta_{\tau,1}(z))^{\alpha-1}e^{-\tau(z+\rho U(x_0))}G^0dz.
\end{equation}

Next, we provide a lemma about $\beta_{\tau,1}(z)$ defined in \eqref{betatau}.
\begin{lemma}[\cite{Deng:17}]\label{lemmabetatau}
	Let $U(x_0)$ be bounded in $\bar{\Omega}$. By choosing $\theta\in(\frac{\pi}{2},\pi)$ sufficiently close to $\frac{\pi}{2}$ and $\kappa>0$ sufficiently large $($depending on $|\rho|\|U(x_0)\|_{L^{\infty}(\bar{\Omega})}$$)$, there exists a positive constant $\tau_{*}$ $($depending on $\theta$ and $\kappa$$)$ such that the following estimates hold when $\tau\leq \tau_{*}$:
	\begin{enumerate}[(1)]
		\item For ${{z}}\in \Sigma^{\tau}_{\theta,\kappa}$, we have $\beta_{\tau,1}({{z}}) \in \Sigma_{\frac{3\pi}{4},C\kappa}$, and
		\begin{equation*}
		C_1|{{z}}|\leq|\beta_{\tau,1}({{z}})|\leq C_2|{{z}}|.
		\end{equation*}
		
		\item The operator $((\beta_{\tau,1}({{z}}))^\alpha+A)^{-1}:L^2(\Omega)\rightarrow L^2(\Omega)$ is well-defined, bounded, and analytic with respect to $z\in \Sigma^{\tau}_{\theta,\kappa}$, satisfying
		\begin{equation*}
		\|A((\beta_{\tau,1}(z))^\alpha+A)^{-1}\|\leq C~~~~~for~all ~~{{z}} \in \Sigma^{\tau}_{\theta,\kappa},
		\end{equation*}
		\begin{equation*}
		\|((\beta_{\tau,1}({{z}}))^\alpha+A)^{-1}\|\leq C|{{z}}|^{-\alpha}~~~~~for~all~~{{z}} \in \Sigma^{\tau}_{\theta,\kappa},
		\end{equation*}
where $\Sigma^{\tau}_{\theta,\kappa}=\{z\in\mathbb{C}:|z|\geq\kappa, |\arg z|\leq \theta, |Im(z)|\leq \frac{\pi}{\tau}, Re(z)\leq \kappa+1\}$. Here, $Im(z)$ means the imaginary part of $z$ and $Re(z)$ the real part of $z$.
\item For the real number $\gamma$, the following estimate holds
\begin{equation*}
	|(\beta(z))^{\gamma}-(\beta_{\tau,1}(z))^{\gamma}|\leq C\tau|z|^{\gamma+1},\quad z\in\Gamma^{\tau}_{\theta,\kappa}.
\end{equation*}
	\end{enumerate}
\end{lemma}

\begin{theorem}\label{thmsemierror}
	Let $G(x_0,\rho,t)$ and $G^n$  be the solutions of Eqs. \eqref{eqretosoleq} and \eqref{equsemischeme} respectively and assume $G_0\in L^2(\Omega)$. Then we obtain
		\begin{equation*}
	\|G(x_0,\rho,t_n)-G^n\|_{L^2(\Omega)}\leq Ct_n^{-1}\tau\|G_0\|_{L^2(\Omega)}.
	\end{equation*}
\end{theorem}
\begin{proof}
	Subtracting \eqref{equsemisolrep} from the inverse Laplace transform of \eqref{equsolrep}, we have
	\begin{equation*}
		\begin{aligned}
				&\|G(x_0,\rho,t_n)-G^n\|_{L^2(\Omega)}\\
				\leq& C\left \|\int_{\Gamma_{\theta,\kappa}}e^{zt_n}((\beta(z))^{\alpha}+A)^{-1}(\beta(z))^{\alpha-1}G_{0}dz\right .\\
				&\qquad\left .-\int_{\Gamma_{\theta,\kappa}^{\tau}}e^{zt_n}((\beta_{\tau,1}(z))^{\alpha}+A)^{-1}(\beta_{\tau,1}(z))^{\alpha-1}e^{-\tau(z+\rho U(x_0))}G^0dz\right \|_{L^2(\Omega)}\\
				\leq& C\left\|\int_{\Gamma_{\theta,\kappa}\backslash\Gamma_{\theta,\kappa}^{\tau}}e^{zt_n}((\beta(z))^{\alpha}+A)^{-1}(\beta(z))^{\alpha-1}dz\right \|\|G_{0}\|_{L^2(\Omega)}\\
				&+C\Big \|\int_{\Gamma_{\theta,\kappa}^{\tau}}e^{zt_n}\big (((\beta(z))^{\alpha}+A)^{-1}(\beta(z))^{\alpha-1}\\
				&\qquad\qquad-((\beta_{\tau,1}(z))^{\alpha}+A)^{-1}(\beta_{\tau,1}(z))^{\alpha-1}e^{-\tau(z+\rho U(x_0))}\big)dz\Big \|\|G_{0}\|_{L^2(\Omega)}\\
				\leq&C(\uppercase\expandafter{\romannumeral1}+\uppercase\expandafter{\romannumeral2})\|G_{0}\|_{L^2(\Omega)}.
		\end{aligned}
	\end{equation*}
	For $\uppercase\expandafter{\romannumeral1}$, using Lemma \ref{lemmaBeta}, there is
	\begin{equation*}
		\begin{aligned}
			\uppercase\expandafter{\romannumeral1}\leq&\int_{\Gamma_{\theta,\kappa}\backslash\Gamma_{\theta,\kappa}^{\tau}}|e^{zt_n}||z|^{-1}|dz|
			\leq C\tau\int_{\frac{\pi}{\tau\sin(\theta)}}^{\infty}e^{t_nr\cos(\theta)}dr
			\leq Ct_n^{-1}\tau.
		\end{aligned}
	\end{equation*}
	As for $\uppercase\expandafter{\romannumeral2}$, one can split it into
	\begin{equation*}
		\begin{aligned}
			\uppercase\expandafter{\romannumeral2}\leq&\left \|\int_{\Gamma_{\theta,\kappa}^{\tau}}e^{zt_n}\left (((\beta(z))^{\alpha}+A)^{-1}(\beta(z))^{\alpha-1}-((\beta_{\tau,1}(z))^{\alpha}+A)^{-1}(\beta(z))^{\alpha-1}\right)dz\right \|\\
			&+\left \|\int_{\Gamma_{\theta,\kappa}^{\tau}}e^{zt_n}\left (((\beta_{\tau,1}(z))^{\alpha}+A)^{-1}(\beta(z))^{\alpha-1}-((\beta_{\tau,1}(z))^{\alpha}+A)^{-1}(\beta_{\tau,1}(z))^{\alpha-1}\right)dz\right \|\\
			&+\left \|\int_{\Gamma_{\theta,\kappa}^{\tau}}e^{zt_n}((\beta_{\tau,1}(z))^{\alpha}+A)^{-1}(\beta_{\tau,1}(z))^{\alpha-1}\left (1-e^{-\tau(z+\rho U(x_0))}\right)dz\right \|.
		\end{aligned}
	\end{equation*}
	Then by Lemmas \ref{lemmaBeta}, \ref{lemmabetatau} and the fact
	\begin{equation*}
		\begin{aligned}
			&\left \|((\beta(z))^{\alpha}+A)^{-1}(\beta(z))^{\alpha-1}-((\beta_{\tau,1}(z))^{\alpha}+A)^{-1}(\beta(z))^{\alpha-1}\right\|\\
			= &\left \|((\beta(z))^{\alpha}+A)^{-1}((\beta_{\tau,1}(z))^{\alpha}-(\beta(z))^{\alpha})((\beta_{\tau,1}(z))^{\alpha}+A)^{-1}(\beta(z))^{\alpha-1}\right\|
			\leq C\tau,
		\end{aligned}
	\end{equation*}
	it has
	\begin{equation*}
		\begin{aligned}
			\uppercase\expandafter{\romannumeral2}\leq & C\tau \int_{\Gamma^\tau_{\theta,\kappa}}|e^{zt_n}||dz|\leq C\tau t_n^{-1}.
		\end{aligned}
	\end{equation*}
	Thus
	\begin{equation*}
		\|G(x_0,\rho,t_n)-G^n\|_{L^2(\Omega)}\leq Ct_n^{-1}\tau\|G_0\|_{L^2(\Omega)}.
	\end{equation*}
\end{proof}

\subsection{Second-order backward difference scheme and error estimate}
In this subsection, we use second-order backward difference convolution quadrature to discretize the time fractional substantial derivative and obtain the second-order accuracy in time. Similarly, introduce $G^n$ as the numerical approximation of the solution $G(x_0,\rho,t_n)$, and let
\begin{equation}\label{equdefd2n}
\delta_{\tau,2}(\zeta)=\frac{(1-\zeta)+(1-\zeta)^2/2}{\tau }\quad and\quad   \nu(\zeta)=\left(\frac{3-\zeta}{2(1-\zeta)}\right )\zeta=\zeta\left (\frac{3}{2}+\sum\limits_{n=1}^{\infty}\zeta^n\right ).
\end{equation}

According to \eqref{equsolrep}, we have
\begin{equation}\label{equsolrep2}
%\begin{aligned}
%\tilde{G}(z)-(\beta(z))^{-1}G_0=&(((\beta(z))^{\alpha}+A)^{-1}(\beta(z))^{\alpha}-I)(\beta(z))^{-1}G_{0}\\
%=& -((\beta(z))^{\alpha}+A)^{-1}A(\beta(z))^{-1}G_{0},
%\end{aligned}
\tilde{G}(z)-(\beta(z))^{-1}G_0=-((\beta(z))^{\alpha}+A)^{-1}A(\beta(z))^{-1}G_{0}.
\end{equation}
%where $I$ means identity operator. %Thus
%\begin{equation*}
%	\tilde{G}(z)-(\beta(z))^{-1}G_{0}=-((\beta(z))^{\alpha}+A)^{-1}A(\beta(z))^{-1}G_{0}.
%\end{equation*}
 Using $\beta_{\tau,2}(z):=\delta_{\tau,2}(e^{-\beta(z)\tau})$, $\tau\sum\limits_{n=1}^{\infty}(G^n-e^{-t_n\rho U(x_0)}G^0)e^{-zt_n}$, and $\tau\nu(e^{-\beta(z)\tau})$ to approximate $\beta(z)$, $\tilde{G}(z)-(\beta(z))^{-1}G_0$, and $(\beta(z))^{-1}$ respectively, we have
 \begin{equation*}
 	\begin{aligned}
 	\sum_{n=1}^{\infty}(G^{n}-e^{-t_{n}\rho U(x_0)}G^0)e^{-zt_n}=-((\delta_{\tau,2
 	}(e^{-\beta(z)\tau}))^{\alpha}+A)^{-1}A\nu(e^{-\beta(z)\tau})G^0.
 	\end{aligned}
 \end{equation*}
Thus
\begin{equation*}
\begin{aligned}
((\delta_{\tau,2
}(e^{-\beta(z)\tau}))^{\alpha}+A)\sum_{n=1}^{\infty}(G^{n}-e^{-t_{n}\rho U(x_0)}G^0)e^{-zt_n}=-A\nu(e^{-\beta(z)\tau})G^0.
\end{aligned}
\end{equation*}
 By Cauchy's integral formula, there exists the second-order temporal semi-discrete scheme
\begin{equation}\label{equsemischemeO2}
\left\{
\begin{aligned}
&d^{\alpha,2}_{0}e^{-t_{0}\rho U(x_0)}G^{1}+A G^{1}+\frac{1}{2}Ae^{-t_{1}\rho U(x_0)}G^0=e^{-t_{1}\rho U(x_0)}d^{\alpha,2}_{0}G^0,\quad x_0\in \Omega,\\	
&\sum_{i=0}^{n-1}d^{\alpha,2}_{i}e^{-t_{i}\rho U(x_0)}G^{n-i}+A G^n=e^{-t_{n}\rho U(x_0)}\sum_{i=0}^{n-1}d^{\alpha,2}_{i}G^0,\quad\,\, x_0\in \Omega,\quad n=2,\ldots,N,\\	
&G^0=G_0,\qquad\qquad\qquad\qquad\qquad\qquad\qquad\qquad\qquad\quad  x_0\in \Omega,\\
&G^n(x_0)=0,\qquad\qquad\quad\qquad\qquad\qquad\qquad\qquad\qquad\quad x_0\in \partial \Omega,\quad n=1,\ldots,N,
\end{aligned}
\right.
\end{equation}
where
\begin{equation}\label{equdefdO2}
(\delta_{\tau,2}(\zeta))^{\alpha}=\sum_{i=0}^{\infty}d^{\alpha,2}_{i}\zeta^i.
\end{equation}

%\begin{remark}
%	This discretization
%\end{remark}
\iffalse
Then, we have following properties about $\delta_{\tau}(e^{-z\tau})$.

\begin{lemma}[\cite{Deng:17}]\label{lemmaprodeltatau}
	Let $\theta\in(\frac{\pi}{2},\pi)$ sufficiently close to $\frac{\pi}{2}$, the following estimates hold
	\begin{equation*}
	\begin{aligned}
	&C_1|z|\leq \left|\frac{1-e^{-\tau z}}{\tau}\right|\leq C_2|z|, \qquad \forall z\in \Sigma_{\theta},~ |Im(z)|\leq \frac{\pi}{\tau},\\
	&\left|\frac{1-e^{-\tau z}}{\tau}-z\right|\leq C\tau|z|^{2},\qquad \forall z\in \Sigma_{\theta},~ |Im(z)|\leq \frac{\pi}{\tau},\\
	&\left|\left (\frac{1-e^{-\tau z}}{\tau}\right )^\alpha-z^\alpha\right|\leq C\tau|z|^{2},\qquad \forall z\in \Sigma_{\theta},~ |Im(z)|\leq \frac{\pi}{\tau},\\
	&\left(\frac{1-e^{-\tau z}}{\tau}\right)^\alpha\in \Sigma_{\theta},\qquad \forall z\in \Sigma_{\theta},~ |Im(z)|\leq \frac{\pi}{\tau},		
	\end{aligned}
	\end{equation*}	
	where $Im(z)$ means the imaginary part of $z$.
\end{lemma}
\fi
Multiplying $\zeta^1$ and $\zeta^n$ on both sides of the first and second formulas of \eqref{equsemischemeO2} respectively and summing them lead to
\begin{equation}\label{equsemimidO2}
\begin{aligned}
&\sum_{n=1}^{\infty}\sum_{i=0}^{n-1}d^{\alpha,2}_{i}e^{-t_{i}\rho U(x_0)}(G^{n-i}-e^{-t_{n-i}\rho U(x_0)}G^0)\zeta^n+\sum_{n=1}^{\infty}A( G^n-e^{-t_{n}\rho U(x_0)}G^0)\zeta^n\\
&=-\left(\sum_{n=1}^{\infty}Ae^{-t_{n}\rho U(x_0)}G^0\zeta^n+\frac{1}{2}Ae^{-t_{1}\rho U(x_0)}G^{0}\zeta\right).	
\end{aligned}
\end{equation}
According to \eqref{equdefd2n} and \eqref{equdefdO2}, after some simple calculations, we get
%\begin{equation*}
%\begin{aligned}
%\zeta\left (\frac{3}{2}+\sum\limits_{n=1}^{\infty}\zeta^n\right )=&\left (\frac{3}{2}-\frac{\zeta}{2}\right )\sum\limits_{n=1}^{\infty}\zeta^n\\
%=&\left (\frac{3-\zeta}{2(1-\zeta)}\right )\zeta=\nu(\zeta)
%\end{aligned}
%\end{equation*}
%and simple calculation imply
\begin{equation*}
\begin{aligned}
& (\delta_{\tau,2}(e^{-\tau\rho U(x_0)}\zeta))^{\alpha}\sum_{n=1}^{\infty}(G^{n}-e^{-t_{n}\rho U(x_0)}G^0)\zeta^n+A\sum_{n=1}^{\infty}(G^{n}-e^{-t_{n}\rho U(x_0)}G^0)\zeta^n
\\
&
=-A\nu(e^{-\tau\rho U(x_0)}\zeta)G^0,
\end{aligned}
\end{equation*}
which can be further written as
\begin{equation*}
\sum_{n=1}^{\infty}(G^{n}-e^{-t_{n}\rho U(x_0)}G^0)\zeta^n=-((\delta_{\tau,2
}(e^{-\tau\rho U(x_0)}\zeta))^{\alpha}+A)^{-1}A\nu(e^{-\tau\rho U(x_0)}\zeta)G^0.
\end{equation*}
By Cauchy's integral formula, there is
\begin{equation}\label{betatauO2}
\begin{aligned}
G^n-e^{-t_{n}\rho U(x_0)}G^0=&-\frac{1}{2\pi\mathbf{i}}\int_{|\zeta|=\xi_\tau}\zeta^{-n-1}((\delta_{\tau,2}(e^{-\tau\rho U(x_0)}\zeta))^{\alpha}+A)^{-1}A\nu(e^{-\tau\rho U(x_0)}\zeta)G^0d\zeta\\
=&-\frac{\tau}{2\pi\mathbf{i}}\int_{\Gamma^\tau}e^{zt_n}((\beta_{\tau,2}(z))^{\alpha}+A)^{-1}A\nu(e^{-\tau\beta(z)})G^0dz,
\end{aligned}
\end{equation}
where $\xi_\tau=e^{-\tau(\kappa+1)}$, $\Gamma^\tau=\{z=\kappa+1+\mathbf{i}y:y\in\mathbb{R}~{\rm and}~|y|\leq \pi/\tau\}$, $\beta_{\tau,2}(z)=\delta_{\tau,2}(e^{-\beta(z)\tau})$, and the second equality follows by taking $\zeta=e^{-z\tau}$. Introduce $\mu(\zeta)=\tau\delta_{\tau,2}(\zeta)\nu(\zeta)=\frac{\zeta(3-\zeta)^2}{4}$ and deform the contour $\Gamma^\tau$ to
$\Gamma^\tau_{\theta,\kappa}=\{z\in \mathbb{C}:\kappa\leq |z|\leq\frac{\pi}{\tau\sin(\theta)},|\arg z|=\theta\}\cup\{z\in \mathbb{C}:|z|=\kappa,|\arg z|\leq\theta\}$. Then there exists
\begin{equation}\label{equsemisolrepO2}
G^n-e^{-t_{n}\rho U(x_0)}G^0=-\frac{1}{2\pi\mathbf{i}}\int_{\Gamma_{\theta,\kappa}^\tau}e^{zt_n}((\beta_{\tau,2}(z))^{\alpha}+A)^{-1}A(\beta_{\tau,2}(z))^{-1}\mu(e^{-\tau\beta(z)})G^0dz.
\end{equation}

Now, we provide a lemma about $\beta_{\tau,2}(z)$ defined in \eqref{betatauO2}.
\begin{lemma}\label{lemmabetatauO2}
	Let $U(x_0)$ be bounded in $\bar{\Omega}$. By choosing $\theta\in(\frac{\pi}{2},\pi)$ sufficiently close to $\frac{\pi}{2}$ and $\kappa>0$ sufficiently large $($depending on $|\rho|\|U(x_0)\|_{L^{\infty}(\bar{\Omega})}$$)$, there exists a positive constant $\tau_{*}$ $($depending on $\theta$ and $\kappa$$)$ such that the following estimates hold when $\tau\leq \tau_{*}$:
	\begin{enumerate}[(1)]
		\item For ${{z}}\in \Sigma^{\tau}_{\theta,\kappa}$, we have $\beta_{\tau,2}({{z}}) \in \Sigma_{\frac{3\pi}{4},C\kappa}$, and
		\begin{equation*}
		C_1|{{z}}|\leq|\beta_{\tau,2}({{z}})|\leq C_2|{{z}}|.
		\end{equation*}
		
		\item The operator $((\beta_{\tau,2}({{z}}))^\alpha+A)^{-1}:L^2(\Omega)\rightarrow L^2(\Omega)$ is well-defined, bounded, and analytic with respect to $z\in \Sigma^{\tau}_{\theta,\kappa}$, satisfying
		\begin{equation*}
		\|A((\beta_{\tau,2}(z))^\alpha+A)^{-1}\|\leq C~~~~~for~all ~~{{z}} \in \Sigma^{\tau}_{\theta,\kappa},
		\end{equation*}
		\begin{equation*}
		\|((\beta_{\tau,2}({{z}}))^\alpha+A)^{-1}\|\leq C|{{z}}|^{-\alpha}~~~~~for~all~~{{z}} \in \Sigma^{\tau}_{\theta,\kappa},
		\end{equation*}
		where $\Sigma^{\tau}_{\theta,\kappa}=\{z\in\mathbb{C}:|z|\geq\kappa,|\arg z|\leq \theta, |Im(z)|\leq \frac{\pi}{\tau},Re(z)\leq \kappa+1\}$. Here, $Im(z)$ means the imaginary part of $z$ and $Re(z)$ the real part of $z$.
		\item For the real number $\gamma$, there is
		\begin{equation*}
		|(\beta(z))^{\gamma}-(\beta_{\tau,2}(z))^{\gamma}|\leq C\tau^2|z|^{\gamma+2},\quad z\in\Gamma^{\tau}_{\theta,\kappa}.
		\end{equation*}
	\end{enumerate}
\end{lemma}
%\begin{proof}
%	The proof is similar to the one in \cite{Deng:17}.
%\end{proof}

\begin{proof}
	First there are the facts \cite{Deng:17}:
	\begin{equation}\label{equlemO2eq2}
	\begin{aligned}
	&if\, \tau|\beta(z)|\geq C,~then~~ |\tau\delta_{\tau,2}(e^{-\tau\beta(z)})|\geq C;\\
	&\tau|{{z}}|\leq \tau|{Im}({{z}})|+\tau|Re(z)|\leq 2\pi,\quad for~z\in \Sigma^{\tau}_{\theta,\kappa}~and~\tau\leq \frac{\pi}{\kappa+1};\\
	&\tau|\beta(z)|\leq \tau|{{z}}|+\tau|{{\rho}}|\|{{U}}(x_0)\|_{L^{\infty}(\bar{\Omega})}\leq \frac{5}{2}\pi,\quad for~z\in \Sigma^{\tau}_{\theta,\kappa}~and~\tau\leq \frac{\pi}{\kappa+1}.\\
	\end{aligned}
	\end{equation}
	Then we prove the boundedness of $|\beta_{\tau,2}(z)|$. Choosing $\kappa\geq2|\rho|\|{{U}}(x_0)\|_{L^{\infty}(\bar{\Omega})}$ and using Taylor's expansion yield that, for $z\in\Sigma^{\tau}_{\theta,\kappa}$,
	\begin{equation*}
	\left|\beta_{\tau,2}({{z}})\right|=\left|\delta_{\tau,2}(e^{-\tau\beta(z)})\right| \leq C|\beta(z)| \leq C\left(|{{z}}|+|{{\rho}}|\|{{U}}(x_0)\|_{L^{\infty}(\bar{\Omega})}\right) \leq C(|{{z}}|+\kappa) \leq C|z|,
	\end{equation*}
	where the fact  $|z|\geq \kappa$ for $z\in\Sigma^{\tau}_{\theta,\kappa}$ is used. Thus the inequality $|\beta_{\tau,2}(z)|\leq C|z|$ holds.
	
	Next, we prove $C|{{z}}|\leq |\beta_{\tau,2}({{z}})|$ for ${{z}}\in \Sigma^{\tau}_{\theta,\kappa}$ in two cases.
	If $\tau|\beta(z)|$ is smaller than some constant, then we use Taylor's expansion (with $|\mathcal{O}(\tau\beta(z))|<\frac{1}{2}$, due to the smallness of $\tau|\beta(z)|$ assumed):
	\begin{equation*}
	\begin{aligned}
	\left|\beta_{\tau,2}({{z}})\right|=\left|\frac{3/2-2e^{-\tau \beta(z)}+e^{-2\tau \beta(z)}/2}{\tau}\right|&=|\beta(z)(1+\mathcal{O}(\tau^2 (\beta(z))^2))| \geq \frac{1}{2}|\beta(z)| \\
	& \geq \frac{1}{2}\left(|{{z}}|-|{{\rho}}|\|{{U}}(x_0)\|_{L^{\infty}(\bar{\Omega})}\right) \\
	& \geq \frac{1}{2}(|{{z}}|-\kappa / 2) \geq \frac{1}{4}|{{z}}|, \end{aligned}
	\end{equation*}
	where we have used $\kappa\geq 2|{{\rho}}|\|{{U}}(x_0)\|_{L^{\infty}(\bar{\Omega})}$ again and that $|z|\geq \kappa$ for $z\in \Sigma^{\tau}_{\theta,\kappa}$.
	
	If $\tau|\beta(z)|$ is larger than the constant, then  \eqref{equlemO2eq2} implies
	\begin{equation*}
	|\beta_{\tau,2}(z)|=\left |\delta_{\tau,2}(e^{-\tau\beta(z)})\right |\geq \frac{C}{\tau}\geq C|z|.
	\end{equation*}
	Thus, under the conditions $\kappa\geq2|\rho|\|{{U}}(x_0)\|_{L^{\infty}(\bar{\Omega})}$ and $\tau<\frac{\pi}{\kappa+1}$, we have proved that
	\begin{equation*}
	C_{1}|{{z}}| \leq\left|\beta_{\tau,2}({{z}})\right|=\left|\delta_{\tau,2}(e^{-\tau\beta(z)})\right| \leq C_{2}|{{z}}| \quad \forall {{z}} \in \Sigma^{\tau}_{\theta, \kappa},
	\end{equation*}
which leads to (according to Lemma \ref{lemO200} provided in the following)
	\begin{equation*}
	\beta_{\tau,2}(z)\in\Sigma_{\pi/2+\epsilon}.
	\end{equation*}
From $\beta_{\tau,2}(z)\geq C|z|$, we have $\beta_{\tau,2}\in \Sigma_{\pi/2+\epsilon,C|\kappa|}$, which results in the second conclusion of this Lemma by using the resolvent estimate \cite{Jin2016}.

As for the third conclusion, there is
	\begin{equation*}
	\begin{aligned}
	&|(\beta(z))^{\gamma}-(\beta_{\tau,2}(z))^{\gamma}|\\
	=&\left |(\beta(z))^{\gamma}-\left(\beta(z)+\tau^2(\beta(z))^3\int_{0}^{1}(1-s)^2e^{-s\tau\beta(z)}ds-2\tau^2(\beta(z))^3\int_{0}^{1}(1-s)^2e^{-2s\tau\beta(z)}ds\right)^{\gamma}\right |\\
	=&|(\beta(z))^\gamma|\left |1-\left (1+\tau^2(\beta(z))^{2}\int_{0}^{1}(1-s)^2e^{-s\tau\beta(z)}ds-2\tau^2(\beta(z))^2\int_{0}^{1}(1-s)^2e^{-2s\tau\beta(z)}ds\right )^{\gamma}\right |.
	\end{aligned}
	\end{equation*}
	If $\tau|\beta(z)|\leq1/2$, by Taylor's expansion, we have
	\begin{equation*}
	\left |1+\tau^2(\beta(z))^{2}\int_{0}^{1}(1-s)^2e^{-s\tau\beta(z)}ds-2\tau^2(\beta(z))^2\int_{0}^{1}(1-s)^2e^{-2s\tau\beta(z)}ds\right |^{\gamma}=1+\mathcal{O}(\tau^2|\beta(z)|^2).
	\end{equation*}
	So,
	\begin{equation*}
	|(\beta(z))^{\gamma}-(\beta_{\tau,2}(z))^{\gamma}|\leq|\beta(z)|^{\gamma}C\tau^2|\beta(z)|^{2}= C\tau^2|\beta(z)|^{\gamma+2}.
	\end{equation*}
	As for $\tau|\beta(z)|>1/2$, we have
	\begin{equation*}
	\begin{aligned}
	&\tau|z|\geq C\tau|\beta_{\tau,2}(z)|\geq C,\qquad \forall z\in\Gamma^{\tau}_{\theta,\kappa},\\
	&|(\beta(z))^{\gamma}-(\beta_{\tau,2}(z))^{\gamma}|\leq C|z|^{\gamma}\leq C\tau^2|z|^{\gamma+2},\qquad \forall z\in\Gamma^{\tau}_{\theta,\kappa}.\\
	\end{aligned}
	\end{equation*}
	Thus the third conclusion is reached.
\end{proof}

Next we provide a lemma about $\delta_{\tau,2}(e^{-z\tau})$ defined in \eqref{equdefd2n}.

\begin{lemma}\label{lemO200}
	Let $U(x_{0})$ be bounded in $\bar{\Omega}$  and $L=|{{\rho}}|\|U({{x_{0}}})\|_{L^{\infty}(\bar{\Omega})}$. There exist positive constants $\theta_{0}\in\left (\frac{\pi}{2},\frac{5\pi}{8}\right )$, $\tau_0$, and $c_0$ such that if $\theta\in\left(\frac{\pi}{2},\theta_0\right )$ and $\tau\in (0,\tau_0]$, then
	\begin{equation}\label{equlem00O2con}
		\begin{aligned}
			\delta_{\tau,2}(e^{-z\tau})\in \Sigma_{\pi/2+\epsilon}, \quad if ~~|{{z}}|\neq 0,~~|\arg({{z}})|\leq \theta,~~and~~|{Im}({{z}})|\leq \frac{\pi}{\tau}+L.
		\end{aligned}
	\end{equation}
	Here $Im(z)$ means the imaginary part of $z$.
\end{lemma}
\begin{proof}
	Obviously, if $|{{z}}|\neq 0$ and $\arg({{z}})=0$, then we have $\arg \left(\delta_{\tau,2}(e^{-z\tau})\right)=0$.
	
	If $|{{z}}|\neq 0$, $\arg({{z}})=\varphi\in(0,\theta]$, and $0\leq {Im}(z)\leq \pi/\tau+L$, then $\omega=\tau|{{z}}|\sin(\varphi)\in(0,\pi+L\tau]$ and it's easy to see that
	\begin{itemize}
		\item[(1)] if $\omega\in(0,\pi]$, then $\arg \left(\delta_{\tau,2}(e^{-z\tau})\right) \in[0, \pi)$;
		\item[(2)] if $\omega\in(\pi,\pi+L\tau]$, then there exists a constant $c_0$ such that $\arg \left(\delta_{\tau,2}(e^{-z\tau})\right) \in[-c_0\tau, 0)$.
	\end{itemize}
	
	For (2), the conclusion can be directly obtained.
	
	For (1), if $\omega=\pi$, then $\arg(\delta_{\tau,2}(e^{-z\tau}))=0$ and \eqref{equlem00O2con} holds. Introduce $\varphi=\arg(z)$ and $r=e^{-\tau|z|\cos(\varphi)}$. When $\omega\in (0,\pi)$ and $\varphi\leq\frac{\pi}{2}$, i.e., $r\leq 1$, then we have
	\begin{equation*}
		\begin{aligned}
			\cot \left(\arg \left(\delta_{\tau,2}(e^{-z\tau})\right)\right) &=\frac{3/2-2r\cos(\omega)+r^2\cos(2\omega)/2}{2r\sin(\omega)-r^2\sin(2\omega)/2} \\
			&=\frac{3/2-2r\cos(\omega)+r^2(2\cos^{2}(\omega)-1)/2}{2r\sin(\omega)-r^2\sin(\omega)\cos(\omega)} \\
			%&=\frac{\left(\frac{3}{2}-\frac{r^{2}}{2}\right)+\cos (\omega)\left(r^{2} \cos (\omega)-2 r\right)}{2 r \sin (\omega)-r^{2} \sin (\omega) \cos (\omega)} \\
			&=\frac{\left(\frac{3}{2}-\frac{r^{2}}{2}\right)+\cos (\omega)\left(r^{2} \cos (\omega)-2 r\right)}{\sin (\omega)\left(2 r-r^{2} \cos (\omega)\right)} \\ &=\frac{\left(\frac{3}{2}-\frac{r^{2}}{2}\right)}{\sin (\omega)\left(2 r-r^{2} \cos (\omega)\right)}-\frac{\cos (\omega)}{\sin (\omega)}\\
			&\geq\frac{1}{\sin (\omega)\left(2 - \cos (\omega)\right)}-\frac{\cos (\omega)}{\sin (\omega)}\\
			&\geq \frac{1-(2-\cos(\omega))\cos(\omega)}{\sin (\omega)\left(2 - \cos (\omega)\right)}\geq 0.
		\end{aligned}
	\end{equation*}
	Thus $\arg \left(\delta_{\tau,2}(e^{-z\tau})\right)\leq \pi/2$.
	
	When $\varphi>\pi/2$, choosing $\varphi$ close to $\pi/2$, using Lemma \ref{lemmabetatauO2} and the definitions of $\delta_{\tau,2}$ and $\omega$, we have
	\begin{equation*}
		\tau|\delta_{\tau,2}^{\prime}(e^{-\sigma|z|\tau\cos(\varphi)}e^{-\mathbf{i}\omega})|\leq C \quad and\quad |\delta_{\tau,2}(e^{-\mathbf{i}\omega})|\geq C|z|,\quad\sigma\in[0,1],
	\end{equation*}
	where $\delta_{\tau,2}^{\prime}(z)$ means the first derivative about $z$.
	Thus we obtain
	\begin{equation*}
		\begin{aligned}
			|\delta_{\tau,2}(e^{-|z|\tau\cos(\varphi)}e^{-\mathbf{i}\omega})-\delta_{\tau,2}(e^{-\mathbf{i}\omega})|\leq& C\left |e^{-\sigma|z|\tau\cos(\varphi)}\right |\left|\delta_{\tau,2}^{\prime}(e^{-\sigma|z|\tau\cos(\varphi)}e^{-\mathbf{i}\omega})z\tau\cos(\varphi)\right| \\
			\leq& C|\cos(\varphi)||\delta_{\tau,2}(e^{-\mathbf{i}\omega})|\\
			\leq& C\left|\varphi-\frac{\pi}{2}\right||\delta_{\tau,2}(e^{-\mathbf{i}\omega})|.
		\end{aligned}
	\end{equation*}
	Using the fact \cite{Lubich1988} that when $|\zeta|\leq 1$ and $\zeta\neq 0$, $\delta_{\tau,2}(\zeta)\in \Sigma_{\pi/2}$ holds, we have $\delta_{\tau,2}(e^{-z\tau})$ lies in a sector $\Sigma_{\pi/2+\epsilon}$.

	So, we have proved \eqref{equlem00O2con} when $\arg(z)\in[0,\theta]$. The case $\arg(z)\in [-\theta,0)$ can be proved in the same way.
\end{proof}

\begin{theorem}\label{thmsemierrorO2}
	Let $G(x_0,\rho,t)$ and $G^n$  be the solutions of Eqs. \eqref{eqretosoleq} and \eqref{equsemischemeO2} respectively and assume $G_0\in L^2(\Omega)$. Then there exist
	\begin{equation*}
	\|G(x_0,\rho,t_n)-G^n\|_{L^2(\Omega)}\leq Ct_n^{-2}\tau^{2}\|G_0\|_{L^2(\Omega)}.
	\end{equation*}
\end{theorem}
\begin{proof}
	Subtracting \eqref{equsemisolrepO2} from the inverse Laplace transform of \eqref{equsolrep2} leads to
	\begin{equation*}
	\begin{aligned}
	&\|G(x_0,\rho,t_n)-G^n\|_{L^2(\Omega)}\\
	\leq& C\left \|\frac{1}{2\pi\mathbf{i}}\int_{\Gamma_{\theta,\kappa}}e^{zt_n}((\beta(z))^{\alpha}+A)^{-1}A(\beta(z))^{-1}G_0dz\right .\\
	&\qquad\qquad\left .-\frac{1}{2\pi\mathbf{i}}\int_{\Gamma_{\theta,\kappa}^\tau}e^{zt_n}((\beta_{\tau,2}(z))^{\alpha}+A)^{-1}A(\beta_{\tau,2}(z))^{-1}\mu(e^{-\tau\beta(z)})G^0dz\right \|_{L^2(\Omega)}\\
	\leq& C\left\|\int_{\Gamma_{\theta,\kappa}\backslash\Gamma_{\theta,\kappa}^{\tau}}e^{zt_n}((\beta(z))^{\alpha}+A)^{-1}A(\beta(z))^{-1}dz\right \|\|G_{0}\|_{L^2(\Omega)}\\
	&+C\Big \|\int_{\Gamma_{\theta,\kappa}^{\tau}}e^{zt_n}\big (((\beta(z))^{\alpha}+A)^{-1}A(\beta(z))^{-1}  \\ &
-((\beta_{\tau,2}(z))^{\alpha}+A)^{-1}A(\beta_{\tau,2}(z))^{-1}\mu(e^{-\tau\beta(z)})\big)dz\Big \|\|G_{0}\|_{L^2(\Omega)}\\
	\leq&C(\uppercase\expandafter{\romannumeral1}+\uppercase\expandafter{\romannumeral2})\|G_{0}\|_{L^2(\Omega)}.
	\end{aligned}
	\end{equation*}
	For $\uppercase\expandafter{\romannumeral1}$, using Lemma \ref{lemmaBeta}, it has
	\begin{equation*}
	\begin{aligned}
	\uppercase\expandafter{\romannumeral1}\leq&\int_{\Gamma_{\theta,\kappa}\backslash\Gamma_{\theta,\kappa}^{\tau}}|e^{zt_n}||z|^{-1}|dz|
	\leq C\tau^2\int_{\frac{\pi}{\tau\sin(\theta)}}^{\infty}e^{t_nr\cos(\theta)}rdr
	\leq Ct_n^{-2}\tau^2.
	\end{aligned}
	\end{equation*}
	As for $\uppercase\expandafter{\romannumeral2}$, it can be split as
	\begin{equation*}
	\begin{aligned}
	\uppercase\expandafter{\romannumeral2}\leq&\left \|\int_{\Gamma_{\theta,\kappa}^{\tau}}e^{zt_n}\left (((\beta(z))^{\alpha}+A)^{-1}A(\beta(z))^{-1}-((\beta_{\tau,2}(z))^{\alpha}+A)^{-1}A(\beta(z))^{-1}\right)dz\right \|\\
	&+\left \|\int_{\Gamma_{\theta,\kappa}^{\tau}}e^{zt_n}\left (((\beta_{\tau,2}(z))^{\alpha}+A)^{-1}A(\beta(z))^{-1}-((\beta_{\tau,2}(z))^{\alpha}+A)^{-1}A(\beta_{\tau,2}(z))^{-1}\right)dz\right \|\\
	&+\left \|\int_{\Gamma_{\theta,\kappa}^{\tau}}e^{zt_n}((\beta_{\tau,2}(z))^{\alpha}+A)^{-1}A(\beta_{\tau,2}(z))^{-1}\left (1-\mu\left (e^{-\tau\beta_{\tau,2}(z)}\right )\right)dz\right \|.
	\end{aligned}
	\end{equation*}
	Then by Lemmas \ref{lemmaBeta}, \ref{lemmabetatauO2} and the facts
	\begin{equation*}
	\begin{aligned}
	&\left \|((\beta(z))^{\alpha}+A)^{-1}A(\beta(z))^{-1}-((\beta_{\tau,2}(z))^{\alpha}+A)^{-1}A(\beta(z))^{-1}\right\|\\
	= &\left \|((\beta(z))^{\alpha}+A)^{-1}((\beta_{\tau,2}(z))^{\alpha}-(\beta(z))^{\alpha})((\beta_{\tau,2}(z))^{\alpha}+A)^{-1}A(\beta(z))^{-1}\right\|
	\leq C\tau^2|z|
	\end{aligned}
	\end{equation*}
	and $\mu(e^{-z\tau})=1+O(z^2\tau^2)$ \cite{Lubich1996}, we have
	\begin{equation*}
	\begin{aligned}
	\uppercase\expandafter{\romannumeral2}\leq & C\tau^2 \int_{\Gamma^\tau_{\theta,\kappa}}|e^{zt_n}||z||dz|\leq C\tau^2 t_n^{-2}.
	\end{aligned}
	\end{equation*}
	Thus
	\begin{equation*}
	\|G(x_0,\rho,t_n)-G^n\|_{L^2(\Omega)}\leq Ct_n^{-2}\tau^{2}\|G_0\|_{L^2(\Omega)}.
	\end{equation*}
\end{proof}

\section{Space discretization and error analysis}
In this section, we discretize Laplacian by the finite element method and provide the error estimate for the fully discrete scheme of Eq. \eqref{eqretosoleq}. Here we construct the fully discrete scheme based on backward Euler scheme \eqref{equsemischeme}; the corresponding one for \eqref{equsemischemeO2} will be commented at the end of this section. Let $\mathcal{T}_h$ be a shape regular quasi-uniform partitions of the domain $\Omega$, where $h$ is the maximum diameter. Denote $ X_h $ as piecewise linear finite element space
\begin{equation*}
X_{h}=\{v_h\in C(\bar{\Omega}): v_h|_\mathbf{T}\in \mathcal{P}^1,\  \forall \mathbf{T}\in\mathcal{T}_h,\ v_h|_{\partial \Omega}=0\},
\end{equation*}
where $\mathcal{P}^1$ denotes the set of piecewise polynomials of degree $1$ over $\mathcal{T}_h$. Then we define the $ L^2 $-orthogonal projection $ P_h:\ L^2(\Omega)\rightarrow X_h $ by
\begin{equation*}
\begin{aligned}
&(P_hu,v_h)=(u,v_h) ~~~~\forall v_h\in X_h.\\
\end{aligned}
\end{equation*}
\iffalse
The $ L^2 $-orthogonal projection $ P_h $ has the following approximation property.

\begin{lemma}[\cite{Bazhlekova2015}]\label{lemprojection}
	The projection $ P_h $ satisfies
	\begin{equation*}
	\begin{aligned}
	&\|P_hu-u\|_{L^2(\Omega)}+h\|\nabla(P_hu-u)\|_{L^2(\Omega)}\leq Ch^q\|u\|_{\dot{H}^{q}(\Omega)}\ ~~for\ u\in \dot{H}^q(\Omega),\,\ q=1,2.\\
	\end{aligned}
	\end{equation*}
\end{lemma}
\fi
Denote $(\cdot,\cdot)$ as the $L_2$ inner product. Then the fully discrete scheme for Eq. \eqref{eqretosoleq} reads: Find $ G^n_{h}\in X_h$ such that
\begin{equation}\label{equsfullscheme}
\left (\sum_{i=0}^{n-1}d^{\alpha,1}_{i}e^{-t_{i}\rho U(x_0)}G^{n-i}_h,v_{h}\right )+(\nabla G^n_{h}, \nabla v_{h})=\left (e^{-t_{n}\rho U(x_0)}\sum_{i=0}^{n-1}d^{\alpha,1}_{i}G^0,v_h\right )
\end{equation}
for any $v_{h}\in X_h$. For convenience, denote	$f(x_0,\rho,t)=e^{-t\rho U(x_0)}G^0$, $f(t)=f(x_0,\rho,t)$, $f^n=f(t_n)$, and $f^n_h=P_hf^n$ in the following.  Thus \eqref{equsfullscheme} can be rewritten as
\begin{equation}\label{equsfullscheme1}
	\left (\sum_{i=0}^{n-1}d^{\alpha,1}_{i}e^{-t_{i}\rho U(x_0)}G^{n-i}_h,v_{h}\right )+(\nabla G^n_{h}, \nabla v_{h})=\left (\sum_{i=0}^{n-1}d^{\alpha,1}_{i}f^{n}_h,v_h\right )
\end{equation}
for any $v_{h}\in X_h$.

\begin{remark}
	Using the time discretization introduced in Sec. 3, the time semi-discrete scheme of Eq. \eqref{eqretosol} can be written as
	\begin{equation*}%\label{equsemischemeneq}
	\left\{
	\begin{aligned}
	&\frac{G^{n}-G^{n-1}}{\tau}+\sum_{i=0}^{n-1}d^{\alpha,1}_{i}e^{-t_{i}\rho U(x_0)}A G^{n-i}+\rho{U}(x_0)G^{n}=0,\quad x_0\in \Omega,\quad n=1,\ldots,N,\\	
	&G^0=G_0,\qquad\qquad\qquad\qquad\qquad\qquad\qquad\qquad\qquad\quad\,\,\,\,  x_0\in \Omega,\\
	&G^n(x_0)=0,\qquad\qquad\quad\qquad\qquad\qquad\qquad\qquad\qquad\quad\,\,\,\, x_0\in \partial \Omega,\quad n=1,\ldots,N,
	\end{aligned}
	\right.
	\end{equation*}
	and  the fully discrete scheme has the form
	\begin{equation}\label{equfullschemeneq1}
		\left(\frac{G^{n}_{h}-G^{n-1}_{h}}{\tau},v_h\right)-\sum_{i=0}^{n-1}d^{\alpha,1}_{i}\left (e^{-t_{i}\rho U(x_0)}\Delta G^{n-i}_h,v_{h}\right )+( \rho U(x_0)G^n_{h}, v_{h})=0, ~~~~\forall v_h\in X_h,
	\end{equation}
	or
	\begin{equation}\label{equfullschemeneq2}
	\left(\frac{G^{n}_{h}-G^{n-1}_{h}}{\tau},v_h\right)+ \sum_{i=0}^{n-1}d^{\alpha,1}_{i}\left(\nabla G^{n-i}_h,\nabla e^{-t_{i}\rho U(x_0)}v_{h}\right )+( \rho U(x_0)G^n_{h}, v_{h})=0, ~~~~\forall v_h\in X_h.
	\end{equation}
	It is easy to see that the second term $\left (\sum_{i=0}^{n-1}d^{\alpha,1}_{i}e^{-t_{i}\rho U(x_0)}\Delta G^{n-i}_h,v_{h}\right )$ in \eqref{equfullschemeneq1}  vanishes since $G^i_h\in X_h$. As for \eqref{equfullschemeneq2}, we need to require $U(x_0)$ regular enough to guarantee
	\begin{equation*}
		\left (e^{-t_{i}\rho U(x_0)}\Delta G^{n-i}_h,v_{h}\right )=\left(\nabla G^{n-i}_h,\nabla e^{-t_{i}\rho U(x_0)}v_{h}\right ).
	\end{equation*}
	Thus the equivalent form \eqref{eqretosoleq} can help to construct numerical scheme efficiently and reduce the requirement of the regularity of $U(x_0)$.
\end{remark}

Define the discrete operator $A_{h}$: $X_h\rightarrow X_h$ satisfying
\begin{equation*}
(A_{h} u_h,v_h)=(\nabla u_h,\nabla v_h) \quad\forall u_h,\ v_h\in X_h.
\end{equation*}
Then  \eqref{equsfullscheme}  can be rewritten as
\begin{equation}\label{equfullAhscheme} 	
\sum_{i=0}^{n-1}d^{\alpha,1}_{i}e^{-t_{i}\rho U(x_0)}G^{n-i}_h+A_h G^n_{h}=\sum_{i=0}^{n-1}d^{\alpha,1}_{i}f^n_h.
\end{equation}
Multiplying $\zeta^n$ and summing $n$ from $1$ to $\infty$ for Eq. \eqref{equfullAhscheme} lead to
\begin{equation*}
\begin{aligned}
\sum_{n=1}^{\infty}\sum_{i=0}^{n-1}d^{\alpha,1}_{i}e^{-t_{i}\rho U(x_0)}G^{n-i}_h\zeta^n+\sum_{n=1}^{\infty}A_h G^n_{h}\zeta^n=\sum_{n=1}^{\infty}\sum_{i=0}^{n-1}d^{\alpha,1}_{i}f^n_h\zeta^n.
\end{aligned}
\end{equation*}
Taking $\zeta=e^{-z\tau}$ and using Cauchy's integral theorem lead to
\begin{equation}\label{equfullsolrepspac}
	G^n_h=\frac{\tau}{2\pi \mathbf{i}}\int_{\Gamma^\tau_{\theta,\kappa}}e^{zt_n}((\beta_{\tau,1}(z))^{\alpha}+A_h)^{-1}\sum_{n=1}^{\infty}\sum_{i=0}^{n-1}d^{\alpha,1}_{i}f^n_he^{-zt_n}dz.
\end{equation}
Similarly, the solution $G^n$ of semi-discrete scheme \eqref{equsemischeme} can also be written as
\begin{equation}\label{equsemisolrepspac}
	G^n=\frac{\tau}{2\pi \mathbf{i}}\int_{\Gamma^\tau_{\theta,\kappa}}e^{zt_n}((\beta_{\tau,1}(z))^{\alpha}+A)^{-1}\sum_{n=1}^{\infty}\sum_{i=0}^{n-1}d^{\alpha,1}_{i}f^ne^{-zt_n}dz.
\end{equation}

\begin{remark}\label{remfullscheme}
	According to \cite{Bazhlekova2015,Jin2016}, the convergence in space can be obtained by estimating
	\begin{equation*}
		\begin{aligned}
			&\|G^n-G^n_h\|_{L^2(\Omega)}\\
			\leq&C\left \|\int_{\Gamma^\tau_{\theta,\kappa}}e^{zt_n}((\beta_{\tau,1}(z))^\alpha+A)^{-1}(\beta_{\tau,1}(z))^{\alpha-1}G_0dz\right .\\
			&\qquad\left .-\int_{\Gamma^{\tau}_{\theta,\kappa}}e^{zt_n}((\beta_{\tau,1}(z))^\alpha+A_h)^{-1}(\beta_{\tau,1}(z))^{\alpha-1}P_hG_0dz\right \|_{L^2(\Omega)}\\
			\leq&C\int_{\Gamma^\tau_{\theta,\kappa}}e^{zt_n}\Big \|((\beta_{\tau,1}(z))^\alpha+A)^{-1}(\beta_{\tau,1}(z))^{\alpha-1}
\\
			&\qquad -((\beta_{\tau,1}(z))^\alpha+A_h)^{-1}(\beta_{\tau,1}(z))^{\alpha-1}P_h\Big \||dz|\|G_0\|_{L^2(\Omega)},
		\end{aligned}
	\end{equation*}
	where the representation of $G^n_h$ can be got by modifying the fully discrete scheme \eqref{equsfullscheme} as
	\begin{equation}\label{equsfullschemew}
	\left (\sum_{i=0}^{n-1}d^{\alpha,1}_{i}e^{-t_{i}\rho U(x_0)}G^{n-i}_h,v_{h}\right )+(\nabla G^n_{h}, \nabla v_{h})=\left (e^{-t_{n}\rho U(x_0)}\sum_{i=0}^{n-1}d^{\alpha,1}_{i}P_hG^0,v_h\right ).
	\end{equation}
	Since we need to estimate $\|((\beta_{\tau,1}(z))^\alpha+A)^{-1}(\beta_{\tau,1}(z))^{\alpha-1}-((\beta_{\tau,1}(z))^\alpha+A_h)^{-1}(\beta_{\tau,1}(z))^{\alpha-1}P_h\|$, the following equation is needed (refer to the proof of Lemma \ref{lemeroper})
	\begin{equation}\label{equremark}
		((\beta_{\tau,1}(z))^{\alpha}G^0,v_{h})=((\beta_{\tau,1}(z))^{\alpha}P_hG^0,v_{h}) \quad\forall v_{h}\in X_h.
	\end{equation}
	Obviously Eq. \eqref{equremark} holds only when $U(x_0)$ is regular enough and the grid mesh is suitable.
\end{remark}
But for the scheme \eqref{equsfullscheme}, the equality \eqref{equremark} is no longer necessary. Next we introduce two lemmas, which will be used in the error estimate between Eqs. \eqref{equsemischeme} and   \eqref{equsfullscheme}.

\begin{lemma}[\cite{Bazhlekova2015}]\label{lemestofapb}
 For any $\gamma_1, ~\gamma_2>0$ with $\theta\in(\pi/2,\pi)$, there exists
	\begin{equation*}
	\gamma_1|z|+\gamma_2\leq \frac{|\gamma_1 z+\gamma_2|}{\sin(\theta/2)} \ for\ z\in \Sigma_{\pi-\theta}.
	\end{equation*}
\end{lemma}

\begin{lemma}\label{lemeroper}
	Let $v\in L^2(\Omega)$ and $z\in\Sigma_{\theta,\kappa}$. Denote $w=((\beta_{\tau,1}(z))^{\alpha}+A)^{-1}v$ and $w_h=((\beta_{\tau,1}(z))^{\alpha}+A_h)^{-1} P_hv$. There exists
	\begin{equation*}
	\|w-w_h\|_{L^2(\Omega)}+h\|w-w_h\|_{\dot{H}^1(\Omega)}\leq Ch^{2}\|v\|_{L^2(\Omega)}.
	\end{equation*}
\end{lemma}
\begin{proof}
	Using the definitions of $w$, $w_h$, $A$ and $A_h$, we have
	\begin{equation*}
	\begin{aligned}
	&((\beta_{\tau,1}(z))^{\alpha} w,\chi)+(\nabla w,\nabla \chi)=(v,\chi)\quad \forall \chi\in H^1_0(\Omega),\\
	&((\beta_{\tau,1}(z))^{\alpha} w_h,\chi)+(\nabla  w_h,\nabla \chi)=(P_hv,\chi)\quad \forall \chi\in X_h.\\
	\end{aligned}
	\end{equation*}
	Denote $e=w-w_h$. Thus
	\begin{equation}\label{equimprelation}
	((\beta_{\tau,1}(z))^{\alpha} e,\chi)+( \nabla e,\nabla \chi)=0 \quad \forall \chi\in X_h.
	\end{equation}
	Choose $\beta_{\tau,max}=\beta_{\tau,1}(z,x)$ for some $x\in\Omega$ satisfying $|\beta_{\tau,max}|=\sup_{x\in\Omega}|\beta_{\tau,1}(z,x)|$ and $\beta_{\tau,min}=\beta_{\tau,1}(z,x)$ for some $x\in\Omega$ satisfying $|\beta_{\tau,min}|=\inf_{x\in\Omega}|\beta_{\tau,1}(z,x)|$. According to Lemmas \ref{lemmabetatau} and \ref{lemestofapb}, there exists
	\begin{equation}\label{equalityww}
	\begin{aligned}
	|\beta_{\tau,max}^{\alpha}|\|e\|^2_{L^2(\Omega)}+\|\nabla e\|^2_{L^2(\Omega)}\leq& C|\beta_{\tau,min}^{\alpha}|\|e\|^2_{L^2(\Omega)}+\|\nabla e\|^2_{L^2(\Omega)}\\
	\leq& C\left |((\beta_{\tau,1}(z))^{\alpha} e,e)\right |+\|\nabla e\|^2_{L^2(\Omega)}\\
	\leq& C|((\beta_{\tau,1}(z))^{\alpha} e,e)+(\nabla e,\nabla e)|\\
	\leq& C|((\beta_{\tau,1}(z))^{\alpha} e,w-\chi)+(\nabla e,\nabla (w-\chi))|,
	\end{aligned}
	\end{equation}
	the detailed derivations of which can be seen in Appendix A. Taking $\chi=\pi_h w$ as the Lagrange interpolation of $w$ and using the Cauchy-Schwarz inequality, we have
	\begin{equation*}
	|\beta_{\tau,max}^{\alpha}|\|e\|^2_{L^2(\Omega)}+\|\nabla e\|^2_{L^2(\Omega)}\leq C\left ( |\beta_{\tau,max}^{\alpha}|h\|e\|_{L^2(\Omega)}\|\nabla w\|_{L^2(\Omega)}+h\|\nabla e\|_{L^2(\Omega)}\|w\|_{\dot{H}^{2}(\Omega)}\right ).
	\end{equation*}
	Using Lemma \ref{lemestofapb} again, it has
	\begin{equation*}
	\begin{aligned}
	|\beta_{\tau,min}^{\alpha}|\|w\|^2_{L^2(\Omega)}+\|\nabla w\|^2_{L^2(\Omega)}\leq& C|(((\beta_{\tau,1}(z))^{\alpha}+A)w,w)|
	\leq C\|v\|_{L^2(\Omega)}\|w\|_{L^2(\Omega)},
	\end{aligned}
	\end{equation*}
	which leads to
	\begin{equation*}
	\|w\|_{L^2(\Omega)}\leq C|z|^{-\alpha}\|v\|_{L^2(\Omega)},\quad \|\nabla w\|_{L^2(\Omega)}\leq C|z|^{-\alpha/2}\|v\|_{L^2(\Omega)}.
	\end{equation*}
	On the other hand, we can obtain
	\begin{equation*}
	\begin{aligned}
	\|w\|_{\dot{H}^{2}(\Omega)}=&\|Aw\|_{L^2(\Omega)}\leq C\|(-(\beta_{\tau,1}(z))^{\alpha}+(\beta_{\tau,1}(z))^{\alpha}+A)((\beta_{\tau,1}(z))^{\alpha}+A)^{-1}v\|_{L^{2}(\Omega)}\\
	\leq&C(|\|v\|_{L^2(\Omega)}+|z|^{\alpha}\|w\|_{L^2(\Omega)})
	\leq C\|v\|_{L^2(\Omega)}.
	\end{aligned}
	\end{equation*}
	Thus we have
	\begin{equation*}
	\begin{aligned}
	&|\beta_{\tau,max}^{\alpha}|\|e\|^2_{L^2(\Omega)}+\|\nabla e\|^2_{L^2(\Omega)}\\
	\leq& Ch\|v\|_{L^2(\Omega)}\left(|z|^{\alpha/2}\|e\|_{L^2(\Omega)}+\|\nabla e\|_{L^2(\Omega)}\right),
	\end{aligned}
	\end{equation*}
	which leads to
	\begin{equation}\label{eqlemH1}
	|\beta_{\tau,max}|^{\alpha/2}\|e\|_{L^2(\Omega)}+\|\nabla e\|_{L^2(\Omega)}\leq Ch\|v\|_{L^2(\Omega)}.
	\end{equation}
		To get the $L^2$ estimate, for $\phi\in L^2(\Omega)$ we set
	\begin{equation*}
	\psi=((\beta_{\tau,1}(z))^{\alpha}+A)^{-1}\phi.
	\end{equation*}
		Using Lemma \ref{lemmabetatau}, we have
	\begin{equation}\label{equpsi1}
	\|\psi\|_{L^2(\Omega)}\leq C|z|^{-\alpha}\|\phi\|_{L^2(\Omega)},\quad \|A\psi\|_{L^2(\Omega)}\leq C\|\phi\|_{L^2(\Omega)}.
	\end{equation}
	Interpolation property leads to
	\begin{equation}\label{equpsi2}
		\|\psi\|_{\dot{H}^{1}(\Omega)}\leq C|z|^{-\alpha/2}\|\phi\|_{L^2(\Omega)}.
	\end{equation}
	By duality, there is
	\begin{equation*}
	\|e\|_{L^2(\Omega)}\leq \sup_{\phi\in L^2(\Omega)}\frac{|(e,\phi)|}{\|\phi\|_{L^2(\Omega)}}\leq \sup_{\phi\in L^2(\Omega)}\frac{|((\beta_{\tau,1}(z))^{\alpha}e,\psi)+(\nabla e,\nabla \psi)|}{\|\phi\|_{L^2(\Omega)}}.
	\end{equation*}
Furthermore,
	\begin{equation*}
	\begin{aligned}
	|((\beta_{\tau,1}(z))^{\alpha}e,\psi)+(\nabla e,\nabla \psi)|= & |((\beta_{\tau,1}(z))^{\alpha} e,\psi - P_{h}\psi) + (\nabla e,\nabla (\psi - P_{h}\psi ))|\\
	\leq&|z|^{\alpha/2}\|e\|_{L^2(\Omega)}|z|^{\alpha/2}\|\psi -P_{h}\psi\|_{L^2(\Omega)}\\
	&+\|\nabla e\|_{L^2(\Omega)}\|\nabla (\psi - P_{h}\psi )\|_{L^2(\Omega)}\\
	\leq& Ch^{2} \|v\|_{L^2(\Omega)}\|\phi\|_{L^2(\Omega)},
	\end{aligned}
	\end{equation*}
	which follows by Eqs. \eqref{eqlemH1}, \eqref{equpsi1}, and \eqref{equpsi2}. Thus,
	the desired estimate is obtained.
\end{proof}

\begin{theorem}\label{thmfullerror}
		Let $G^n$ and $G^n_h$  be the solutions of Eqs. \eqref{equsemischeme} and \eqref{equsfullscheme} respectively and assume $G_0\in L^2(\Omega)$. Then we obtain
	\begin{equation*}
	\|G^n-G^n_h\|_{L^2(\Omega)}+h\|\nabla(G^n-G^n_h)\|_{L^2(\Omega)}\leq Ch^2t_n^{-\alpha}\|G_0\|_{L^2(\Omega)}.
	\end{equation*}
\end{theorem}
\begin{proof}
	Subtracting \eqref{equfullsolrepspac} from \eqref{equsemisolrepspac} yields
	\begin{equation*}
		\begin{aligned}
			& \|G^n-G^n_h\|_{L^2(\Omega)} \\
&  \leq C\left\|\tau\int_{\Gamma^\tau_{\theta,\kappa}}e^{zt_n}(((\beta_{\tau,1}(z))^{\alpha}+A)^{-1}-((\beta_{\tau,1}(z))^{\alpha}+A_h)^{-1}P_h)\sum_{j=1}^{\infty}\sum_{i=0}^{j-1}d^{\alpha,1}_{i}f^je^{-zt_j}dz\right\|_{L^2(\Omega)}\\
			&  \leq C\tau\int_{\Gamma^\tau_{\theta,\kappa}}|e^{zt_n}|\left\|((\beta_{\tau,1}(z))^{\alpha}+A)^{-1}-((\beta_{\tau,1}(z))^{\alpha}+A_h)^{-1}P_h\right\|\\
			&\qquad\left \|\sum_{j=1}^{\infty}\sum_{i=0}^{j-1}d^{\alpha,1}_{i}G_0e^{-t_j(z+\rho U(x_0))}\right  \|_{L^2(\Omega)}|dz|\\
	&		\leq C\int_{\Gamma^\tau_{\theta,\kappa}}|e^{zt_{n}}|\left\|((\beta_{\tau,1}(z))^{\alpha}+A)^{-1}-((\beta_{\tau,1}(z))^{\alpha}+A_h)^{-1}P_h\right\|\left\|\beta_{\tau,1}(z)\right\|^{\alpha-1}|dz|\|G_0\|_{L^2(\Omega)}.\\
		\end{aligned}
	\end{equation*}
 Thus by Lemma \ref{lemeroper}, there is
	\begin{equation*}
		\|G^n-G^n_h\|_{L^2(\Omega)}\leq Ch^2t_n^{-\alpha}\|G_0\|_{L^2(\Omega)}.
	\end{equation*}
	Similarly, it can be obtained that
		\begin{equation*}
	\begin{aligned}
& \|G^n-G^n_h\|_{\dot{H}^1(\Omega)}
\\
%\\
&\leq C\left\|\tau\int_{\Gamma^\tau_{\theta,\kappa}}e^{zt_n}(((\beta_{\tau,1}(z))^{\alpha}+A)^{-1}-((\beta_{\tau,1}(z))^{\alpha}+A_h)^{-1}P_h)\sum_{j=1}^{\infty}\sum_{i=0}^{j-1}d^{\alpha,1}_{i}f^je^{-zt_j}dz\right\|_{\dot{H}^1(\Omega)}\\
&	\leq  C\tau\int_{\Gamma^\tau_{\theta,\kappa}}|e^{zt_n}|\left\|((\beta_{\tau,1}(z))^{\alpha}+A)^{-1}-((\beta_{\tau,1}(z))^{\alpha}+A_h)^{-1}P_h\right\|_{L^2(\Omega)\rightarrow \dot{H}^1(\Omega)}\\
	&\qquad\qquad\left \|\sum_{j=1}^{\infty}\sum_{i=0}^{j-1}d^{\alpha,1}_{i}G_0e^{-t_j(z+\rho U(x_0))}\right  \|_{L^2(\Omega)}|dz|\\
	& \leq C\int_{\Gamma^\tau_{\theta,\kappa}}|e^{zt_{n}}|\left\|((\beta_{\tau,1}(z))^{\alpha}+A)^{-1}-((\beta_{\tau,1}(z))^{\alpha}+A_h)^{-1}P_h\right\|_{L^2(\Omega)\rightarrow \dot{H}^1(\Omega)}\\
	&\qquad\qquad\left\|\beta_{\tau,1}(z)\right\|^{\alpha-1}|dz|\|G_0\|_{L^2(\Omega)}\\
	&\leq Cht^{-\alpha}\|G_0\|_{L^2(\Omega)}.
	\end{aligned}
	\end{equation*}
\end{proof}

\begin{remark}
Comparing Theorem \ref{thmfullerror} and Remark \ref{remfullscheme}, it is easy to see that our numerical scheme \eqref{equsfullscheme} needs much less regularity requirement of $U(x_0)$ than scheme \eqref{equsfullschemew} to keep the optimal convergence rate.
\end{remark}

\begin{remark}
	In this section, we provide the complete error analysis based on the time first order scheme \eqref{equsemischeme}. Similarly, from the second order scheme \eqref{equsemischemeO2}, the fully discrete scheme can be written as: for any $v_h\in X_h$,
	\begin{equation}\label{equfullschemeO2}
		\begin{aligned}
			&\left (e^{-t_{0}\rho U(x_0)}G^{1}_h,v_{h}\right )+(\nabla G^1_{h}, \nabla v_{h})+\frac{1}{2}(\nabla P_h(e^{-t_1\rho U(x_0)}G^{0}), \nabla v_{h})
\\
&=\left (P_h(e^{-t_{1}\rho U(x_0)}d^{\alpha,2}_{0}G^0),v_h\right ), \quad n=1;\\
			&\left (\sum_{i=0}^{n-1}d^{\alpha,2}_{i}e^{-t_{i}\rho U(x_0)}G^{n-i}_h,v_{h}\right )+(\nabla G^n_{h}, \nabla v_{h})
\\
&=\sum_{i=0}^{n-1}d^{\alpha,2}_{i}\left (P_h(e^{-t_{n}\rho U(x_0)}G^0),v_h\right ), \quad n=2,\ldots,N.
		\end{aligned}
	\end{equation}
The optimal convergence of the scheme \eqref{equfullschemeO2} can be similarly got by the techniques used in this section.
	\end{remark}
\section{Numerical experiments}
In this section, we first present four numerical experiments to validate the predicted convergence rate of our numerical schemes and then provide an example to show the difference between \eqref{equsfullscheme} and \eqref{equsfullschemew}. Due to the unknown exact solution, the spatial errors can be tested by
\begin{equation*}
\begin{aligned}
E_{h}=\|G^{n}_{h}-G^{n}_{h/2}\|_{L^2(\Omega)},
\end{aligned}
\end{equation*}
where $G^n_{h}$ means the numerical solution of $G$ at time $t_n$ with mesh size $h$; similarly, the temporal errors can be tested by
\begin{equation*}
\begin{aligned}
E_{\tau}=\|G^n_{\tau}-G^n_{\tau/2}\|_{L^2(\Omega)},
\end{aligned}
\end{equation*}
where $G^n_{\tau}$ are the numerical solutions of $G$ at the fixed time $t_n$ with time step size $\tau$. And corresponding convergence rates can be calculated by
\begin{equation*}
{\rm Rate}=\frac{\ln(E_{h}/E_{h/2})}{\ln(2)},\quad {\rm Rate}=\frac{\ln(E_{\tau}/E_{\tau/2})}{\ln(2)}.
\end{equation*}
For convenience, we choose the domain $\Omega=(0,1)$.

Here, the following two groups of initial values and $U(x_0)$ are used:
\begin{enumerate}[(a)]
	\item \begin{equation*}\label{ini05}
		G_0(x_0)=\chi_{(0,1/2)},\quad U(x_0)=\chi_{(1/2,1)};
	\end{equation*}
	\item \begin{equation*}\label{compare}
	G_0(x_0)=\chi_{(0,1/2)},\quad U(x_0)=x_0,
	\end{equation*}
\end{enumerate}
where $\chi_{(a,b)}$ denotes the characteristic function on $(a,b)$.

\begin{example}
In this example, we use backward Euler scheme \eqref{equsemischeme} to solve Eq. \eqref{eqretosoleq} under condition \eqref{ini05} and show the corresponding numerical results. Here we take $T=1$, $\rho=1+\mathbf{i}$, $\alpha=0.3$, $0.7$, and $\tau=1/50$, $1/100$, $1/200$, $1/400$. We use small spatial mesh size $h = 1/128$ so that the spatial discretization error is relatively negligible. Table \ref{tab:atemporal} presents the $L^2$ errors and convergence rates, which agree with Theorem \ref{thmsemierror}.

\begin{table}[htbp]
	\centering
	\caption{$L^2$ errors and convergence rates with condition \eqref{ini05} using numerical scheme \eqref{equsemischeme}}
	\begin{tabular}{c|cccc}
		\hline
		$\alpha\backslash \tau$& 1/50  & 1/100 & 1/200 & 1/400 \\
		\hline
		0.3   & 6.435E-05 & 3.201E-05 & 1.597E-05 & 7.974E-06 \\
		& Rate  & 1.0072  & 1.0036  & 1.0018  \\
		\hline
		0.7   & 1.118E-04 & 5.521E-05 & 2.743E-05 & 1.368E-05 \\
		& Rate  & 1.0180  & 1.0089  & 1.0045  \\
		\hline
	\end{tabular}%
	\label{tab:atemporal}%
\end{table}%
\end{example}

\begin{example}
	In this example, we use second-order backward difference scheme \eqref{equsemischemeO2} to solve Eq. \eqref{eqretosoleq} under condition \eqref{ini05}. Here we take $T=1$,  $\rho=-1+\mathbf{i}$, $\alpha=0.3$, $0.7$, and $\tau=1/10$, $1/20$, $1/40$, $1/80$.	To investigate the convergence in time
	and eliminate the influence from spatial discretization, we set $h = 1/128$. Table \ref{tab:atemporalo2} presents the $L^2$ errors and convergence rates, which agree with Theorem \ref{thmsemierrorO2}.
	\begin{table}[htbp]
		\centering
		\caption{$L^2$ errors and convergence rates with condition \eqref{ini05} using numerical scheme \eqref{equsemischemeO2}}
		\begin{tabular}{c|cccc}
			\hline
			$\alpha\backslash\tau$&    1/10  &    1/20  &    1/40  &    1/80  \\
			\hline
			0.3   & 5.185E-05 & 1.192E-05 & 2.875E-06 & 7.110E-07 \\
			& Rate  & 2.1215  & 2.0515  & 2.0154  \\
			\hline
			0.7   & 1.452E-04 & 3.343E-05 & 7.981E-06 & 1.968E-06 \\
			& Rate  & 2.1190  & 2.0665  & 2.0199  \\
			\hline
		\end{tabular}%
		\label{tab:atemporalo2}%
	\end{table}%
	
\end{example}

\begin{example}
	We consider the spatial convergence of numerical scheme \eqref{equsfullscheme1} under condition \eqref{compare} and present the numerical results. Here we take $T=1$, $\rho=2+\mathbf{i}$, $\alpha=0.2$, $0.8$, and $h=1/16$, $1/32$, $1/64$, $1/128$. To investigate the convergence in spatial
and eliminate the influence from temporal discretization, we set $\tau = 1/1000$. Table \ref{tab:bspatial} presents the $L^2$ errors and convergence rates and Table \ref{tab:bspatial1} provides the $H^1$ errors and convergence rates, which both agree with the results of Theorem \ref{thmfullerror}.
\begin{table}[htbp]
	\centering
	\caption{$L^2$ errors and convergence rates with condition \eqref{compare} using numerical scheme \eqref{equsfullscheme1}}
	\begin{tabular}{c|cccc}
		\hline
		$\alpha\backslash h$&    1/16  &    1/32  &    1/64  &    1/128 \\
		\hline
		0.2   & 1.072E-04 & 2.683E-05 & 6.708E-06 & 1.677E-06 \\
		& Rate  & 1.9988  & 1.9997  & 1.9999  \\
		\hline
		0.8   & 3.151E-05 & 7.885E-06 & 1.972E-06 & 4.929E-07 \\
		& Rate  & 1.9989  & 1.9997  & 1.9999  \\
		\hline
	\end{tabular}%
	\label{tab:bspatial}%
\end{table}%
\begin{table}[htbp]
	\centering
	\caption{$H^1$ errors and convergence rates with condition \eqref{compare} using numerical scheme \eqref{equsfullscheme1}}
	\begin{tabular}{c|cccc}
		\hline
		$\alpha\backslash h$&    1/16  &    1/32  &    1/64  &    1/128 \\
		\hline
		0.2   & 6.062E-03 & 3.033E-03 & 1.517E-03 & 7.586E-04 \\
		& Rate  & 0.9987  & 0.9997  & 0.9999  \\
		\hline
		0.8   & 1.673E-03 & 8.371E-04 & 4.186E-04 & 2.093E-04 \\
		& Rate  & 0.9991  & 0.9998  & 0.9999  \\
		\hline
	\end{tabular}%
	\label{tab:bspatial1}%
\end{table}%

\end{example}

\begin{example}
	We consider the spatial convergence of numerical scheme \eqref{equfullschemeO2} under condition \eqref{ini05} and show the numerical results. Here we take $T=1$, $\rho=-2+\mathbf{i}$, $\alpha=0.4$, $0.6$, and $h=1/10$, $1/20$, $1/40$, $1/80$.  We use small time step $\tau = 1/1000$ so that the temporal discretization error is relatively negligible. Table \ref{tab:bspatialO2} presents the $L^2$ errors and convergence rates and Table \ref{tab:bspatial1O2} provides the $H^1$ errors and convergence rates. These results show that the numerical scheme achieves optimal convergence rates in $L^2$- and $H^1$-norm.
	
	\begin{table}[htbp]
		\centering
		\caption{$L^2$ errors and convergence rates with condition \eqref{ini05} using numerical scheme \eqref{equfullschemeO2}}
		\begin{tabular}{c|cccc}
			\hline
		$\alpha\backslash h$	&    1/10  &    1/20  &    1/40  &    1/80  \\
		\hline
			0.4   & 1.296E-04 & 3.247E-05 & 8.159E-06 & 2.062E-06 \\
			& Rate  & 1.9966  & 1.9927  & 1.9842  \\
			\hline
			0.6   & 9.379E-05 & 2.355E-05 & 5.944E-06 & 1.517E-06 \\
			& Rate  & 1.9934  & 1.9864  & 1.9707  \\
			\hline
		\end{tabular}%
		\label{tab:bspatialO2}%
	\end{table}%
	
	\begin{table}[htbp]
		\centering
		\caption{$H^1$ errors and convergence rates with condition \eqref{ini05} using numerical scheme \eqref{equfullschemeO2}}
		\begin{tabular}{c|cccc}
			\hline
			$\alpha\backslash h$&    1/10  &    1/20  &    1/40  &    1/80  \\
			\hline
			0.4   & 7.296E-03 & 3.648E-03 & 1.824E-03 & 9.120E-04 \\
			& Rate  & 1.0000  & 1.0000  & 1.0000  \\
			\hline
			0.6   & 5.132E-03 & 2.566E-03 & 1.283E-03 & 6.415E-04 \\
			& Rate  & 0.9999  & 1.0000  & 1.0000  \\
			\hline
		\end{tabular}%
		\label{tab:bspatial1O2}%
	\end{table}%

\end{example}

\begin{example}
In this example, we want to verify the effectiveness of our numerical scheme \eqref{equsfullscheme1}. We show the numerical results for solving Eq. \eqref{eqretosoleq} under condition \eqref{ini05}. Here we take $T=1$, $\tau=1/1000$, $\rho=-1+\mathbf{i}$, $\alpha=0.4$, $0.6$, and $h=1/16$, $1/32$, $1/64$, $1/128$. Table \ref{tab:aspatialr} and Table \ref{tab:aspatialw} present the $L^2$ errors and convergence rates of numerical schemes \eqref{equsfullscheme1} and \eqref{equsfullschemew} respectively. Comparing these results, it can be found that the errors of numerical scheme \eqref{equsfullschemew} are much bigger and it can't achieve optimal convergence rates.  These results show that our scheme is effective.
\begin{table}[htbp]
	\centering
	\caption{$L^2$ errors and convergence rates with condition \eqref{ini05} using numerical scheme \eqref{equsfullscheme1}}
	\begin{tabular}{c|cccc}
		\hline
			$\alpha\backslash h$&    1/16  &    1/32  &    1/64  &    1/128 \\
		\hline
		0.4   & 1.296E-04 & 3.239E-05 & 8.097E-06 & 2.024E-06 \\
		& Rate  & 2.0001  & 2.0000  & 2.0000  \\
		\hline
		0.6   & 9.277E-05 & 2.319E-05 & 5.798E-06 & 1.450E-06 \\
		& Rate  & 1.9999  & 2.0000  & 2.0000  \\
		\hline
	\end{tabular}%
	\label{tab:aspatialr}%
\end{table}%

\begin{table}[htbp]
	\centering
	\caption{$L^2$ errors and convergence rates with condition \eqref{ini05} using numerical scheme \eqref{equsfullschemew}}
	\begin{tabular}{c|cccc}
		\hline
		$\alpha\backslash h$&    1/16  &    1/32  &    1/64  &    1/128 \\
		\hline
		0.4   & 9.232E-04 & 4.620E-04 & 2.315E-04 & 1.160E-04 \\
		& Rate  & 0.9988  & 0.9967  & 0.9976  \\
		\hline
		0.6   & 6.282E-04 & 3.150E-04 & 1.580E-04 & 7.919E-05 \\
		& Rate  & 0.9959  & 0.9950  & 0.9968  \\
		\hline
	\end{tabular}%
	\label{tab:aspatialw}%
\end{table}%

\end{example}

\section{Conclusion}
Backward fractional Feynman-Kac equation describes the functional distribution of anomalous diffusion process. The challenge of regularity analysis and numerical analysis mainly comes from its time-space coupled fractional substantial derivative. This work weakens the regularity requirement of the function $U(x)$ in obtaining the optimal convergence rates.  The provided numerical schemes do not need to make the  assumptions on the regularity of the exact solution in temporal and spatial directions. The error estimates are presented with optimal convergence rates. The performed numerical experiments not only verify the theoretical predictions but also show the effectiveness of the techniques introduced in the scheme to keep the accuracy.

%
%We use the convolution quadrature with generating function given by the backward Euler method and second-order backward difference method to approximate the Riemann-Liouville fractional substantial derivative and, respectively, get the first-order and
%second-order approximations and finite element method to approximate the Laplace operator. Compared with the previous work for solving the backward fractional Feynman-Kac equation, our discretization scheme doesn't need the assumption on the regularity of the exact solution in temporal and spatial directions. Moreover, we provide the error estimates for the time semi-discrete and the fully discrete schemes and we get the optimal convergence order. Finally, the numerical experiments verify the effectiveness of our discretization, of which we not only validate the predicted convergence order but also show the  technique of designing our numerical scheme to keep the accuracy.

\section*{Acknowledgements}
This work was supported by the National Natural Science Foundation of China under grant no. 11671182, and the Fundamental Research Funds for the Central Universities under grant no. lzujbky-2018-ot03.

\appendix
\section{Derivation of \eqref{equalityww}}
Without loss of generality, we denote $\Omega=[0,1]$, $0=x_0<x_1<\cdots<x_n=1$, and $\Delta x_i=x_{i}-x_{i-1}$. Let $\theta=\sup_{x\in\Omega}\arg((\beta_{\tau,1}(z,x))^{\alpha})-\inf_{x\in\Omega}\arg((\beta_{\tau,1}(z,x))^{\alpha})$. Then we have
\begin{equation*}
	\begin{aligned}
		\left| ((\beta_{\tau,1}(z,x))^{\alpha}e,e) \right|=\left |\int_{0}^{1}(\beta_{\tau,1}(z,x))^\alpha e(x)^2dx\right |
		=& \lim_{n\rightarrow\infty}\left |\sum_{i=1}^{n}\Delta x_i(\beta_{\tau,1}(z,x_i))^\alpha e(x_i)^2\right |\\
		\geq &\cos\left (\frac{\theta}{2}\right )\lim_{n\rightarrow\infty}\sum_{i=1}^{n}\Delta x_i e(x_i)^2 |(\beta_{\tau,1}(z,x_i))^\alpha|\\
		\geq&\cos\left (\frac{\theta}{2}\right )\lim_{n\rightarrow\infty}\sum_{i=1}^{n}\Delta x_i e(x_i)^2 |\beta_{\tau,min}^\alpha|\\
		\geq&\cos\left (\frac{\theta}{2}\right ) |\beta_{\tau,min}^\alpha|\|e\|^2_{L^2(\Omega)}.
	\end{aligned}
\end{equation*}
Therefore,
\begin{equation*}
	\cos\left (\frac{\theta}{2}\right ) |\beta_{\tau,min}^\alpha|\|e\|^2_{L^2(\Omega)}\leq\left| ((\beta_{\tau,1}(z,x))^{\alpha}e,e) \right|.	
\end{equation*}

\end{document}